\documentclass[12pt]{article}

\usepackage{amscd}
\usepackage{amssymb,amsmath,amsthm}

\theoremstyle{plain}
\newtheorem{thm}{Theorem}[section]
\newtheorem{lem}[thm]{Lemma}
\newtheorem{cor}[thm]{Corollary}
\newtheorem{defn-lem}[thm]{Definition-Lemma}

\newtheorem{prop}[thm]{Proposition}

\theoremstyle{definition}
\newtheorem{defn}[thm]{Definition}
\newtheorem{ex}[thm]{Example}
\newtheorem{rem}[thm]{Remark}

\def\ses#1#2#3{ $$
      \begin{CD}
      0 @>>> #1 @>>>#2 @>>> #3 @>>> 0
      \end{CD}
       $$  }

\def\mc {\mathcal}
\def\mb {\mathbb}
\def\Ind{\operatorname{Ind}}
\def\dim{\operatorname{dim}}
\def\rank{\operatorname{rank}}
\def\Span{\operatorname{span}}
\def\maxi{\operatorname{max}}

\def\Tr{\operatorname{Tr}}

\begin{document}

\title {Homotopy classification of projections \\
       in the corona algebra of a non-simple $C\sp *$-algebra }

\author {Lawrence G. Brown \\
          Department of Mathematics\\
          Purdue University\\
          West Lafayette, USA 47907
         \and
          Hyun Ho \quad Lee \thanks{Corresponding author. \emph{E-mail address}:hyun.lee7425@gmail.com}\\
         Department of mathematical science\\
          Seoul National University\\
         Seoul, Korea 151-747
         }

\date{May 15, 2010}
\maketitle
\begin{abstract}
We study projections in the corona algebra of $C(X)\otimes K$ where $X=[0,1],[0,\infty),(-\infty,\infty)$, and $[0,1]/\{ 0,1 \}$. Using BDF's essential codimension, we determine conditions for a projection in the corona algebra to be liftable to a projection in the multiplier algebra. Then we also characterize the conditions for two projections to be equal in $K_0$-group, Murray-von Neumann equivalent, unitarily equivalent, and homotopic from the weakest to the strongest.  In light of these characterizations, we construct examples showing that
any two equivalence notions do not coincide, which serve as examples of non-stable K-theory of $C\sp*$-algebras.
\end{abstract}
\section{Introduction}\label{Intro}

Our goal of the paper is to study projections in the corona
algebra of a non-simple stable rank one $C\sp{*}$-algebra. The work
presented here was originated from a lifting problem: Let $A$ be a
$C\sp{*}$-algebra and $I$ be a closed ideal of $A$. We are interested whether every unitary
in $A/I$ is liftable to a partial isometry in $A$.  It happens whenever
$I$ has an approximate identity of projections, a weaker condition
than real rank zero. We are concerned with the case $I$ has stable
rank one. It might not be possible in general, but constructing an
explicit counter-example is still not so trivial. One way to do is
to find a stably projectionless stable rank one algebra $I$ such
that $K_0(I)$ is non-trivial. Then stabilize $I$ and consider an
extension of $I$ by $C(\mathbb{T})$ which comes from  a unitary
$\mathbf{u}$ in the corona algebra of $I$ with non-trivial $K_1$-class. i.e.,
\ses{I}{A}{C(\mb{T})}
We can observe that $\mathbf{u}$ can't be lifted to a unitary (if so,
$[\mathbf{u}]=0$ which is contradiction) and it can't be lifted to a
partial isometry either, because there aren't any non-zero
projections available to  be the defect projections of a partial
isometry.

If we denote the $C\sp{*}$-algebra of compact
operators on a separable infinite dimensional Hilbert space by $K$, the cone or suspension of $K$ are stably projectionless but their $K_0$-groups are trivial. Let $D$ be the cone or suspension of $K$ and suppose that we have a projection
$\mathbf{p}$ in the corona algebra of $D$ denoted by $\mc{C}(D)$, which does not lift (stably) but its
$K_0$-class does lift. If we let $a$
be a self adjoint element which lifts $\mathbf{p}$ in the multiplier algebra $M(D)$, we
take $I$ to be the $C\sp{*}$-algebra generated by $a$ and $D$ so
that the quotient $I/D$ is isomorphic to $\mathbb{C}$. Then the
Busby invariant is determined by sending $1$ to $\mathbf{p}$, and we
have the following commutative diagram:
\begin{equation*}
\begin{CD}
   0 @>>> D @>j>> I @>>> \mathbb{C} @>>> 0 \\
    @. @| @VVV @VV\mathbf{p}V \\
   0 @>>> D @>>>M(D) @>>> \mc{C}(D) @>>> 0
\end{CD}
\end{equation*}
  By the long exact sequence, we have
  \begin{equation*}
 \begin{CD}
  K_0(D) @>>>K_0(I) @>>> K_0(\mathbb{C}) @>\partial_0>> K_1(D)  \\
      @. @VVV   @VV\mathbf{p}V @| \\
  K_0(D) @>>>K_0(M(D)) @>>> K_0(\mc{C}(D)) @>\partial_0>> K_1(D)
\end{CD}
\end{equation*}
Since $\partial_0([\mathbf{p}]_0)=0$, $\partial_0:K_0(\mathbb{C})
\to K_1(D)$ becomes trivial. Thus $K_0(I) \cong K_0(\mathbb{C})$. In
particular, $K_0(I)$ is non-trivial. Consequently, we found a
(stably) projectionless stable rank one $C\sp{*}$-algebra such that
 its $K_0$-group is non-trivial. This leads
us to the question when such a projection $\mathbf{p}$ in the corona algebra of $D$ exists.
We note that an element $\mathbf{f}$ in the corona algebra of $D$ is represented by a finite partition and operator valued functions corresponding to finite intervals such that they differ by compact at a partition point. This approach is proven to be useful when $\mathbf{f}$ is a projection since each function corresponding to a closed interval is projection valued. Thus a projection in the corona algebra of $D$ is \emph{locally liftable} in the above sense and their discontinuities are described by pairs of projections in $B(H)$ such that whose difference is compact. Further, we can quantify their discontinuities using a tool so called the \emph{essential codimension}. This quantification allows us not only to solve lifting problem but also to give conditions for homotopy equivalence, unitary
equivalence, Murray-von Neumann equivalence, and $K_0$-equivalence of two projections in the corona algebra.

This paper is arranged as follows: In Section \ref{BDF}, we review the notion of \emph{essential codimension} of two projections in $B(H)$ and derive some facts which will be needed later using Halmos' technique of Hilbert space
decomposition. (In fact, the definition of essential codimension was
given in \cite{bdf} and some properties were provided without proofs
in \cite{br2}. Here we give a more precise definition and provide complete proofs of its properties.)
In Section \ref{Lifting}, we give a necessary and sufficient condition for
the liftability of a projection in the corona algebra of $I=C(X)\otimes
K$ where $X=(-\infty,\infty ), [0,\infty)$, $[0,1]$, $\mb{T}$. In Section \ref{Homotopy}, we give criterions for homotopy equivalence  $\sim_h$, unitary
equivalence $ \sim_u$, and Murray-von Neumann equivalence $\sim$ of two projections
$\mathbf{p}, \mathbf{q}$ in $C(I)\otimes M_n$ for some
$n$ denoting the corona algebra by $C(I)$. In addition, we clarify the condition for the liftability of
$K_0$-class of a projection, i.e., we clarify when it becomes trivial
in $K_0$-group. Thus we construct a projection which dose not
lift but its $K_0$-class does lift if applicable.  Also, we construct
examples such that $[p]_0=[q]_0$ in $K_0$ but $p \nsim q $, $p \sim
q $ but $p \nsim_{u} q$, and $p \sim_{u} q$ but $p \nsim_h q$.
In Section \ref{S:Appendix}, we review some rudiments of continuous fields of
Hilbert spaces and prove theorems and propositions which are crucial
ingredients for the lifting problem of projection in the corona
algebra and other equivalence relations as well.

 \section{Essential codimension}\label{BDF}
\begin{defn}(Brown, Douglas, and Fillmore)
 When $p, q$ are projections in $B(H)$ such that $p-q \in K$, we
 define the \textbf{essential codimension} of $p$ and $q$ which will be denoted
 as $[p:q]$. Let
 $V$, $W$ be isometries such that $VV^*=q,WW^*=p$ if $p$ and $q$ are infinite
 rank. Then
 \begin{equation*}
     [p:q]=
     \begin{cases}
     \Ind(V^*W) &\text{if $p, q$ have infinite rank,} \\
     \rank(p) - \rank(q) &\text{if $p,q$ have finite rank.}
     \end{cases}
 \end{equation*}
 \end{defn}
 In this definition, when $p, q$ have infinite rank, $[p:q]$ is independent of
 the choice of $V$ and $W$. In fact, if we have isometries $V_1, V_2$ such that
 $V_{1}V_{1}^*=q$ and $V_{2}V_{2}^*=q$, then it is easy to check that
 $U=V_{2}^*V_1$ is a unitary, $UV_{1}^*W=V_{2}^*W$. It follows that $\Ind(V_{1}^*W)=\Ind(V_{2}^*W)$. The other case is
 proved similarly.
 \begin{prop}\label{P:properties}
 $[\,:\,]$ has the following properties.

 \begin{itemize}
 \item [(1)] If $p_2 \leq p_1$, then $[p_1:p_2]$ is the usual
 codimension of $p_2$ in $p_1$ which is $\rank(p_1-p_2)$
 \item [(2)]$[p_1:p_2]= -[p_2:p_1]$
 \item [(3)]$[p_1:p_3]=[p_1:p_2]+[p_2:p_3]$
 \item [(4)]$[p_1+p_{1}':p_2+p_{2}']=[p_1:p_2]+[p_{1}':p_{2}']$  \quad when sensible.
 \end{itemize}
\end{prop}
\begin{proof}
 For (1), let $V_i$ be the isometries corresponding to $p_{i}$ for $i=1,2$. Then $V_{2}^*V_{1}$ is a
 co-isometry because $p_2p_1=p_2$. Hence $\Ind(V_{2}^*V_{1})=\dim \ker (V_{2}^*V_{1})= \rank
 (1-(V_{2}^*V_{1})^{*}V_{2}^*V_{1})=
 \Tr(V_{1}^*(p_1-p_2)V_{1})=\Tr(p_1-p_2)=\rank(p_1-p_2)$ since $p_1-p_2 \in
 K$, where $\Tr: L^1(H) \to \mathbb{C}$ is the usual trace map.

 (2) is evident from the definition.

 For (3), if $p_{i}$'s have finite rank, it is easy. If $p_{i}$'s
 have infinite rank, we take the corresponding isometries such that
 $V_{i}V_{i}^*=p_{i}$. Then $V_{3}^{*}V_{2}V_{2}^*V_{1}- V_{3}^{*}V_1 \in
 K$, and therefore $\Ind(V_{3}^*V_1)=\Ind(
 V_{3}^{*}V_{2}V_{2}^*V_{1})=\Ind(V_{3}^*V_{2})+\Ind(V_{2}^*V_{1})$.

 Finally, note that $p_{i}+p_{i}'$ is a projection if and only if $p_{i}$ and
 $p_{i}'$ have orthogonal ranges or $p_{i}p_{i}'=0$. If both $p_{i}$ and
 $p_{i}'$ have finite rank, then
 \begin{equation*}
  \rank(p_{i}+p_{i}')=\rank(p_{i})+\rank(p_{i}') \quad \mbox{for}\quad i=1,2
 \end{equation*}
 It follows that
\begin{align*}
  [p_1+p_{1}':p_2 + p_{2}']  &=\rank(p_{1}+p_{1}')-
  \rank(p_{2}+p_{2}')\\
                                   &=\rank(p_{1})+\rank(p_{1}') - (\rank(p_{2})+
                                   \rank(p_{2}'))\\
                                   &=\rank(p_{1})-\rank(p_{2})+ \rank(p_{1}')-
                                   \rank(p_{2}')\\
                                   &=[p_1:p_2]+[p_{1}':p_{2}'].
\end{align*}
If $p_{i}$ has infinite rank and $p_{i}'$ has finite rank,
then
  $p_{i}+p_{i}'-p_{i} \in K$ for $i=1,2$.
\begin{align*}
   [p_1+p_{1}': p_2 + p_{2}'] &=[p_{1}+p_{1}':p_{1}]+[p_{1}:p_{2}]+[p_{2}:p_{2}+p_{2}'] \\
                                    &=\rank(p_{1}')+[p_{1}:p_{2}]-\rank (p_{2}')\\
                                    &=[p_1:p_2]+[p_{1}':p_{2}'].
\end{align*}
 If both $p_{i}$ and $p_{i}'$ have infinite rank, note that
 $p_{1}'p_{2}, p_{2}'p_{1} \in K$. If we let $U:H \to H \oplus
 H$ be a unitary element, then $V=U^*(V_{1}\oplus V_{1}')U$ and
 $W=U^{*}(V_{2}\oplus V_{2}')U$ are the isometries corresponding to
 $p_{1}+p_{1}'$ and $p_{2}+p_{2}'$ respectively, where $V_{i}$ and $V_{i}'$ are
 isometries such that $V_iV_{i}^{*}=p_{i}$ and
 $V_{i}'{V_{i}'}^{\ast}=p_{i}'$ for $i=1,2$. So
 \begin{align*}
 [p_{1}+p_{1}': p_{2}+p_{2}'] &= \Ind(V^*W)\\
                                    &= \Ind(U^*(V_{1}^*V_{2} \oplus (V_{1}')^*V_{2}' )U) \\
                                    &= \Ind( V_{1}^*V_{2} \oplus (V_{1}')^*V_{2}' ) \\
                                    &= \Ind( V_{1}^*V_{2}) + \Ind ((V_{1}')^*V_{2}') \\
                                    &= [p_1:p_2]+[p_{1}':p_{2}'].
 \end{align*}
 \end{proof}
 \begin{lem}\label{L:unitaryequi}
  Let $p$ and $q$ be projections in $B(H)$ such that $p-q \in K$. If there is a
  unitary $U \in 1+ K$ such that $UpU^*=q$, then $[p:q]=0$. In
  particular,
  if $\| p- q\| <1 $, then $[p:q]=0$.
 \end{lem}
 \begin{proof}
 If ranges of $p$ and $q$ are finite dimensional, $\rank(p)= \rank(UpU^*)=\rank(q)$.
 Now assume $p$ and $q$ are infinite dimensional and let $W$ be the
 isometry such that $WW^*=p$. If there is a unitary $U \in 1 + K$  such that $UpU^*=q$,
  then it is easily checked that $V=UW$ is an isometry such that $VV^*=q$.
 Therefore,
 \begin{align*}
 [p:q] &= \Ind(V^*W) \\
       &= \Ind(W^*UW) \\
       &= \Ind(W^*W+ \mbox{compact})\\
       &= \Ind(I)=0 .
 \end{align*}
 Now if $\|p-q\| <1$, we can take $a=(1-q)(1-p)+qp \in 1+ K$. Since
 $aa^*=a^*a=1-(p-q)^2 \in 1+ K$,

   \[\|a^*a -1\|=\|p-q\|^2 < 1, \quad \|aa^*-1\|=\|p-q\|^2 < 1. \] Moreover, it
   follows that
     \begin{align*}
     ap &= qp =qa.
     \end{align*}
  Hence, $a$ is invertible element and $U=a(a^*a)^{-\frac{1}{2}} \in 1+K$
  is a unitary such that $UpU^*=q$.
\end{proof}

 If a pair of projections $p,q$ such that $p-q \in K$ are given, we can diagonalize $p-q$ by the
 Spectral theorem. We denote by $E_{\lambda}(p-q)$ the eigenspace of $p-q$ corresponding to the eigenvalue $\lambda$.
 Then, relying on a classification of a pair of projections due to Halmos \cite{Ha}, we have the following characterization of $[p:q]$.
 \begin{prop}\label{P:decomposition}
  $$
  [p:q] = \dim E_{1}(p-q) - \dim E_{-1}(p-q).
  $$
 \end{prop}
  \begin{proof}
  Let $M$ and $N$ be the range of $p$ and $q$ respectively, and let
  $H_{11}=M \cap N$, $H_{10}=M \cap N^{\perp}$, $H_{01}=M^{\perp} \cap N $,
  $H_{00}= M^{\perp} \cap N^{\perp}$ and $H_0=(H_{00}\oplus H_{10} \oplus H_{01}
  \oplus H_{11})^{\perp}$. It is possible to identify both $H_0 \cap M$ and $H_0
  \cap M^{\perp}$ with $L^{2}(X)$ for some measure space $X$ in such a way that
  the $p$ on the $H_0$ which is denoted by $p_0$ and the $q$ on the $H_{0}$
  which is denoted by $q_0$ are given by
  \[
     p_0=\begin{pmatrix}
      1 & 0 \\
      0 & 0
     \end{pmatrix} \quad \mbox{and} \quad
     q_0=\begin{pmatrix}
      \cos^2 \phi & \cos \phi \sin \phi \\
      \cos \phi \sin \phi  & \sin^2 \phi
     \end{pmatrix}
  \]
  where $\phi $ is a measurable function on $X$ such that $0< \phi(x) <
  \dfrac{\pi}{2}$ for $x \in X$. Here $p_0$ and $q_0$ operate on $L^{2}(X)
  \oplus L^{2}(X)$ and the matrices are operator matrices whose entries are
  multiplication operators (see \cite[280--281]{br1}).

  If we denote $p$ on the $H_{11}$ by $p_{11}$ and  $p$ on the $H_{10}$ by
  $p_{10}$, then $p=p_{11}+p_{10}+p_{0}$. Similarly,  $q=
  q_{11}+q_{01}+q_{0}$. Now $p-q= p_{10} - q_{01}+p_0 - q_0  \in K$
  implies that $p_0 -q_0 \in K(H_0)$ and $p_{10}-q_{01} \in K(H_{10} \oplus
  H_{01})$. Then
  \[
  p_0-q_0 = \begin{pmatrix}
    \sin^2 \phi & -\cos \phi \sin \phi \\
      -\cos \phi \sin \phi  & -\sin^2 \phi
   \end{pmatrix} \in K(H_0)
  \]
  implies that $X$ is a discrete space $\{ x_n \}$ and $\phi(x_n) \to
  0 \quad \mbox{as} \quad n \to \infty$. Therefore
  \[U= \bigoplus _{n=1}\begin{pmatrix}
      \cos \phi(x_n) & \sin \phi(x_n) \\
      -\sin \phi(x_n) &  \cos \phi(x_n)
     \end{pmatrix} \in 1+K(H_0). \]
   From
  \[
     \begin{pmatrix}
      \cos \phi & -\sin \phi \\
      \sin \phi &  \cos \phi
     \end{pmatrix}     \begin{pmatrix}
      1 & 0 \\
      0 & 0
     \end{pmatrix}
     \begin{pmatrix}
      \cos \phi & \sin \phi \\
      -\sin \phi &  \cos \phi
     \end{pmatrix}= \begin{pmatrix}
      \cos^2 \phi & \cos \phi \sin \phi \\
      \cos \phi \sin \phi  & \sin^2 \phi
     \end{pmatrix},
  \] it follows that $U^*p_0U=q_0$, and $[p_0:q_0]=0$ by  Lemma \ref{L:unitaryequi}.
  On the other hand, $p_{10}-q_{01} \in K(H_{10}\oplus H_{01}) $  means
  $\rank(p_{10})$ and $\rank(q_{01})$ are finite, and
  \begin{align*}
  [p_{10}:q_{01} ] &=\rank(p_{10}) -\rank(q_{01}) \\
                   &=\dim (E_{1}(p-q)) - \dim (E_{-1}(p-q)).
  \end{align*}
   Since $[p:q]=[p_0:q_0]+[p_{10}:q_{01}]+[p_{11}:q_{11}]$ by Proposition \ref{P:properties}, $[p:q] = [p_{10}:q_{01}]$, thus the result follows.
  \end{proof}
 \begin{rem}\label{R:remark1}
 \begin{itemize}
 \item[(i)] In fact, since  $\dim(E_{1}(p-q))$ = \# of 1 in the diagonalization of
 $p-q$ and   $\dim(E_{-1}(p-q))$ = \# of -1 in the diagonalization of
 $p-q$, $[p:q]$ = \mbox{\# of 1 in the diagonalization of $(p-q)$} - \mbox{\# of -1 in
 the diagonalization of
 $(p-q)$}.
 \item[(ii)] The other non-zero points in the spectrum of $p-q$ come from $p_0
 -q_0$. They are $\sin \phi(x_n)$ and $-\sin \phi(x_n)$. Note that this part of
 the spectrum is symmetric about $0$ (even considering multiplicity), i.e., $\dim
  \ker(p_0 -q_0-\sin \phi) =  \dim
 \ker(p_0 -q_0+\sin \phi)$.
  \end{itemize}
 \end{rem}

 \begin{cor}\label{C:homotopy}
 Suppose projections $p_t,  q_t \in B(H)$ are defined for each $t$.  Then
 $[p_t:q_t]$ is constant if $p_t - q_t$ is norm continuous  in $ K$
 \end{cor}
 \begin{proof}
 In fact, we are going to prove that there is a $\delta > 0$ such that whenever
 $A_1=p_1-q_1, A_2=p_2-q_2$ satisfy $\| A_1- A_2\| < \delta$, then
 $[p_1:q_1]=[p_2:q_2]$.

 Since $A_1$ is compact, its spectrum $\sigma(A_1)$ is discrete. So we can take a
 neighborhood $U=O_{-1}(\epsilon) \cup (-1+\epsilon,1-\epsilon) \cup O_{1}(\epsilon) $ containing $\sigma(A_1)$, where
 $O_{\pm1}(\epsilon)$ are open balls with radius $\epsilon $  centered at $\pm 1$
 respectively, for some $\epsilon >0$. By the semicontinuity of spectrum, if
 we take $\delta$ small enough, we know $\sigma(A_2) \subset U$. If we let $\gamma_{\pm1}$ be the
 circles of radius $R$ within $O_{\pm1}(\epsilon)$ centered at $\pm1$ which does not intersect both $\sigma(A_1)$ and $\sigma(A_1)$ , by the Riesz functional calculus, we
 have projections $r_i=\frac{1}{2\pi i}\int_{\gamma_{1}}(z-A_{i})^{-1}dz$ and
 $s_i=\frac{1}{2\pi i}\int_{\gamma_{-1}}(z-A_{i})^{-1}dz$ for $i=1,2$.  Moreover, if we take $\delta$
 small enough, we also have $\|r_{1}- r_{2}\|<1, \|s_{1}-
 s_{2}\|< 1$. Then, by Lemma \ref{L:unitaryequi}, $\rank(r_1)=\rank(r_2)$ and $
 \rank(s_1)=\rank(s_2)$.

 Note that $\rank(s_i)= \sum_{-1 < \lambda
 <-1+R} \,\dim \ker(A_i - \lambda 1) + \dim \ker(A_i+1) $.
 Similarly, $\rank(r_i)= \sum_{1-R < \lambda
 <1} \,\dim \ker(A_i - \lambda 1) + \dim \ker(A_i-1) $. Then
 Remark \ref{R:remark1}-(ii) implies
 \begin{align*}
 [p_1:q_1] &= \dim \ker(A_1-1) - \dim \ker(A_1+1) \\
                     &=\rank(r_1)-\rank(s_1)  \\
             &= \rank(r_2)- \rank(s_2) \\
             &= \dim \ker(A_2-1) - \dim \ker(A_2+1) \\
             &= [p_2:q_2].
 \end{align*}
 \end{proof}
 Now we want to prove the most important property of the essential codimension.
 \begin{thm}\label{T:homotopy}
  Let $p$ and $q$ be projections in $B(H)$ such that $p-q \in K$.  There is a
  unitary $U \in 1+ K$ such that $UpU^*=q$ if and only if $[p:q]=0$.
 \end{thm}
 \begin{proof}
  Lemma \ref{L:unitaryequi} proves one direction.
 For the other direction, suppose that $[p:q]=0$. Using the decomposition of $p-q$ as is
 shown in Proposition \ref{P:decomposition}, we know that $p_0$ and $q_0$ are
 unitarily equivalent, say the unitary $U_0$. In addition, from $[p:q]=\rank(p_{10}-\rank(q_{01})=0$, there
 is a partial isometry $W$ in $B(H_{10}\oplus H_{01})$ such that $W^*W=p_{10}$
 and $WW^*=q_{01}$. Note that $Wp_{10}W^*=q_{01}$ and $W$ is a unitary in
 $1+K(H_{10}\oplus H_{01})$. Now we have a unitary $U=
 U_0+W+1_{H_{11}}+1_{H_{00}} \in 1+K$ such that $UpU^*=q$.
\end{proof}
\begin{rem}
Let $E$ be a Hilbert $B$-module. If $\pi, \sigma: A \to
\mathcal{L}(E)$ are representations, we say that $\pi$ and $\sigma$
are properly asymptotically unitarily equivalent and write $\pi
\approxeq \sigma$ if there is a continuous path of unitaries
$u:[0,\infty) \to \mathcal{U}(K(E) + \mathbb{C}1_E)$, $u=(u_t)_{t\in
[0,\infty)}$ such that
\begin{itemize}
\item $\lim_{t \to \infty}\|u_t\pi(a)u_{t}\sp{*}-\sigma(a) \|=0$
for all $a \in A$
\item $u_t\pi(a)u_t\sp{*}-\sigma(a)\in K(E)$, for all $t\in
[0,\infty)$, and $a\in A.$
\end{itemize}

The use of the word 'proper' reflects the crucial fact that all
implementing unitaries are of the form 'identity + compact' \cite{daei}. A result
of Dadarlat and Eilers shows that if $\phi, \psi: A \to M(B \otimes
K)$ is a Cuntz pair, then $[\phi,\psi]$ vanishes in $KK(A,B)$
 if and only if $\phi\oplus\gamma \approxeq \psi\oplus\gamma$ for
some representation $\gamma:A \to M(B \otimes K)$. As a corollary,
they have shown if $\phi, \psi: A \to B(H)$ is a Cuntz pair of
admissible representations (faithful, non-degenerate, and its image
does not contain any non-trivial compacts), then $[\phi,\psi]$
vanishes in $KK(A,\mathbb{C})$
 if and only if $\phi \approxeq \psi$. Now we apply this to $A=\mathbb{C}$.
  Without loss of generality we can assume $p$ and
 $q$, which come from a Cuntz pair in $KK(\mathbb{C},\mathbb{C})$, are very close if $[p:q]=0$.
 Then  $z=pq+(1-p)(1-q)\in 1+K $ is invertible and
 $pz=zq$. If we consider the polar decomposition of $z$ as $z=u|z|$.
 It is easy to check that  $u \in 1+K$ and $upu^{\ast}=q$. Thus
 Theorem \ref{T:homotopy} can be obtained from a $K$-theoretic
result.
\end{rem}
 In the next section, we need the following facts.
 \begin{prop}\label{P:restrictions}
 Let $p, q$ be projections such that $p-q \in K$.
 \begin{itemize}
 \item [(1)] If $q$ has finite
 rank, then $[p:q] \geq -\rank(q)$.
 \item [(2)] If $1-q$ has finite rank, then $[p:q] \leq \rank(1-q) $
 \end{itemize}
 \end{prop}
 \begin{proof}
  Straightforward.
 \end{proof}
\section{Lifting projections}\label{Lifting}
Let $X$ be $[0,1], [0,\infty)$, $(-\infty,\infty)$ or
$\mathbb{T}=[0,1]/\{0,1\}$ and let $I=C(X)\otimes
   K$ which is the $C\sp{*}$-algebra of (norm continuous) functions from $X$ to $K$. Then $M(I)$ is given by $C_b(X, B(H))$ which is the set of
   bounded functions from $X$ to $B(H)$ where $B(H)$ is given the double-strong
   topology. Let $\mathcal{C}(I)=M(I)/I$ be the corona algebra of $I$ and also
   let $\pi:M(I) \to \mathcal{C}(I)$ be the natural quotient map.

 In general, an element $\mathbf{f}$ of the corona algebra $\mathcal{C}(I)$ can be represented as follows:  Consider a
 finite partition of $X$, or $X \smallsetminus \{0,1\}$ when $X=\mathbb{T}$ given by partition points $x_1 < x_2 < \cdots
 < x_n $ all of which are in the interior of $X$ and divide $X$ into
 $n+1$ (closed) subintervals $X_0,X_1,\cdots,X_{n}$. We can take $f_i \in
 C_b(X_i,B(H))$ such that $f_i(x_i) -f_{i-1}(x_i)\in K$
 for $i=1,2,\cdots,n$ and $f_0(x_0)-f_n(x_0) \in K$  where $x_0=0=1$ if $X$ is $\mathbb{T}$. The coset in $\mathcal{C}(I)$ represented by
 $(f_0,\cdots,f_n)$  consists of functions $f$ in $M(I)$ such that $f- f_i \in
 C(X_i, K)$ for every $i$ and $f-f_i $ vanishes (in norm) at any
 infinite end point of $X_i$. Then $(f_0,\cdots,f_n)$ and $(g_0,\cdots,g_n)$
 define the same element of $\mathcal{C}(A)$ if and only if $f_i - g_i \in
 C_b(X_i,K)$ for $i=0,\cdots,n$ and $f_0-g_0 \in
 C_0((-\infty,x_1],K)$ $f_n -g_n \in
 C_0([x_n,\infty),K)$ if $X$ contains infinite points.

 In particular, if  $\mathbf{f}$ is a projection, we can find $(f_0,f_1,\dots,f_n)$ such that each $f_i$ is projection valued as Theorem \ref{T:partition} shows.  Thus it makes sense to say a projection $\mathbf{f}$ is \emph{locally liftable}.

 \begin{lem}[Calkin]\label{L:calkin}
 Let $T$ be a self-adjoint element of $B(H)$ such that $T-T^{2} \in
 K$ with $\|T\| \leq 1$.
 Then $\sigma(T) \subset E \cup F $ where $E$ and $F$  are two disjoint borel sets containing 0 and 1 respectively,
 and if we let  $P_F(T)$(or $P_1(T)$) be the spectral projection of $T$
 corresponding to the $F$, then
 $ P_{1}(T) - T  \in K$.
 \end{lem}
 \begin{proof}
 See \cite{ca}.
 \end{proof}

\begin{thm}\label{T:partition}
  If $\mathbf{f}$ is a projection in $\mathcal{C}(I)$, we can find an
 $(f_0,\cdots,f_n)$ as above such that each $f_i$ is projection valued.
 \end{thm}
 \begin{proof}
 Let $f$ be the element of $M(I)$ such that $\pi(f)=\mathbf{f}$.
 Without loss of generality, we can assume $f$ is self-adjoint and $0 \leq f \leq 1$.
 \begin{itemize}
 \item[(i)]  Suppose $X$ does not contain any infinite point.
 Choose a point $t_0 \in X $. Then there is a self-adjoint element $T \in B(H)$ such that $T-f(t_0) \in K$ and
 the spectrum of $T$ has a gap around $1/2$. So we consider $f(t)+T-f(t_0)$ which is still a self-adjoint operator whose image is $\mathbf{f}$.
 Thus we may assume $f(t_0)$ is a self-adjoint element whose spectrum has a gap around $1/2$.

 Since $r(f(t)):t \to f(t)-f(t)^{2}$ is norm continuous where $r(x)=x-x^2$,
 if we pick a point $z$ in $\left( 0,\frac{1}{4}\right)$ such that $z \notin \sigma(f(t_0)-f(t_0)^2)$,
 then  $\sigma(f(s))$ omits $r^{-1}(J)$  for $s$ sufficiently
 close to $t$ where $J$ is an interval containing $z$. In other
 words, there is  $\delta>0$ and $b> a >0 $ such that if $|t_0-s|<\delta $, then $ \sigma(f(s)) \subset
 [0, a)\cup (b,1]$. \\
  If we let $f_{t_0}(s)=\chi_{(b,1]}(f(s))$ for $s$ in $(t_0 - \delta,
  t_0+\delta)$ where $\chi_{(b,1]}$ is the characteristic function on $(b,1]$, then it is a continuous projection valued function such that
  $f_{t_0} - f \in C(t_0-\delta,t_0+\delta)\otimes K$.

  By repeating the above procedure, since $X$ is compact, we can find $n+1$ points $t_0,\cdots, t_n$, $n+1$ functions $f_{t_0}, \cdots,  f_{t_n}$,
  and an open covering $\{O_i\}$ such that $t_i \in O_i$, $O_i \cap O_{i-1} \ne
  \varnothing$, and  $f_{t_i}$ is projection valued function on $O_i$. Now
   let $f_i=f_{t_i}$ as above. Take the point $x_i \in O_{i-1} \cap O_i$ for
  $i=1, \cdots, n$. Then $f_{i}(x_i) - f_{i-1}(x_i)= f_{i}(x_i) - f(x_i) +
  f(x_i)- f_{i-1}(x_i) \in K$ and $f_0(x_0)-f_n(x_0) \in K$ if applicable. Let $X_{i}=[x_{i},x_{i+1}]$ for $i=1,
  \cdots, n-1$, $X_0=[0,x_1]$, and $X_n=[x_n,1]$. Since each $f_i$ is also
  defined on $X_{i}$, $(f_0, \cdots, f_n)$ is what we
  want.
  \item[(ii)] let $X$ be $[0, \infty)$.  Since $f^2(t) - f(t) \to 0 $ as $t$
  goes to $\infty$, for given $ \delta$ in $(0, 1/2) $, there is $M >0$ such that
  whenever $t \geq M$ then $\|f^2(t)-f(t) \| < \delta -\delta^2 $. It follows
  that $\sigma(f(t)) \subset [0, \delta) \cup (1-\delta,1]$ for $t \geq M$. Then
  again $P_1(f(t))$ is a continuous projection valued function for $t \geq M$
  such that $f(t) - E_1(f(t))$ vanishes as $t$ goes to $\infty$. . By applying
  the argument in $(i)$ to $[0,M]$, we get a closed subintervals $X_i$ for $i=0,
  \cdots, n-1$ of $[0,M]$ and $f_{i} \in C_{b}(X_{i}, B(H))$. Now if we let
  $X_n=[M, \infty]$ and $f_n(t) = P_1(f(t)) $, we are done.
  \item[(iii)]The case $X=(-\infty, \infty)$ is similar to (ii).
 \end{itemize}
 \end{proof}
Given a representation $(f_0, \cdots, f_n)$ of $\mathbf{f}$ in $\mathcal{C}(I)$, we
  can associate integers $k_i=[f_i(x_i):f_{i-1}(x_i)]$ for $i=1,\cdots, n$ and $k_0=[f_0(x_0):f_n(x_0)]$ in the circle case. The following proposition shows that these numbers are obstructions to global continuity of $(f_0, \cdots, f_n)$.
 \begin{prop}\label{P:lifting}
  If all $k_i$'s are equal to 0, then $\mathbf{f}$ is liftable to a projection in
  $M(I)$.
 \end{prop}
 \begin{proof}
 First we consider the case $X=[0,1]$, $[0,\infty]$, or $(-\infty,\infty)$.
 We recursively construct perturbations $f_1', \cdots f_n'$ of $f_1, \cdots, f_n$ so
 that $f_i'(x_i)=f_{i-1}'(x_i)$ for $i=1 \cdots n$ and  $f_n'$ agrees with
 $f_n$ in the some neighborhood of $\infty$ if the right end point of X is
 $\infty$.

  Observe that, if $k_i=[f_i(x_i):f_{i-1}(x_i)]=0$ for some $i$,
  we can find a
  unitary $U \in 1 + K$ which is path connected to 1 such that
  $Uf_{i}(x_i)U^*=f_{i-1}(x_i)$  by Theorem \ref{T:homotopy}. If we denote the (norm continuous) path of
  unitaries by $U(t)$ such that $U(x_i)=U$, $U(x_{i+1})=1$ and let  $f_{i}'(t)=U(t)
  f_{t}(t){U(t)}^{\ast}$ for each $t
  \in X_{i}$, then $f_{i}'-f_{i} \in C_b(X_{i},K)$ and
  $f_{i}'$ agree at $x_{i}$ with $f_{i-1}$.
  Furthermore, by Corollary \ref{C:homotopy} and Lemma \ref{L:unitaryequi}
  \begin{equation*}
  [f_{i+1}(x_{i+1}):f_{i}'(x_{i+1})] =
  [f_{i+1}(x_{i+1}):f_{i}(x_{i+1})].
  \end{equation*}
  Thus all $k_i$'s are equal to 0,
  we can perturb $f_i$'s for $i=1, \cdots, n$ inductively so that they will still be projection valued
  and will agree at $x_1, \cdots, x_{n}$ in the case $X=[0,1]$. If the right end point of $X$ is
  $\infty$, let $V$ be the unitary such that
  $Vf_{n}(x_{n})V^*=f_{n-1}'(x_n)$, we apply the same process to perturb $f_n$ with the path of
  unitaries defined by
  \begin{equation*}
  U(t)=
        \begin{cases}
         V  &\text{if $t=x_{n}$} \\
        1  &\text{if $t \geq  x_{n} + 1$}.
        \end{cases}
    \end{equation*}
If $X=\mathbb{T}$, we can recursively construct
$f'_1,f'_2,\cdots,f'_n,f'_0$ of $f_1,f_2,\cdots,f_n,f_0$
using the same argument as above since the perturbed map $f'_i$
agrees with $f_i$ at the end point of the interval $X_i$.
\end{proof}

  If $\mathbf{f}$ is liftable to a projection $g$ in $M(I)$,  we
 can use the same partition of $X$ so that $(g_0,\cdots,g_n)$ and $(f_0,
 \cdots, f_n)$ define the same element $\mathbf{f}$ where $g_i$ is the
 restriction of $g$ on $X_i$. Then, for each $i$, $[g_i(x):f_i(x)]$ is
 defined for all $x$. From Corollary \ref{C:homotopy}  this function must be constant on $X_i$
 since $g_i -f_i$ is norm continuous. So we can let $l_i= [g_i(x):f_i(x)]$. \\
 Since $g_i(x_i)=g_{i-1}(x_i)$, we have
 $[g_i(x_i):f_i(x_i)]+[f_i(x_i):f_{i-1}(x_i)]=[g_{i-1}(x_i):f_{i-1}(x_i)]$ by
 Proposition \ref{P:properties}-(3). In other words,
 \begin{equation}\label{E:eq1}
 l_i-l_{i-1}=-k_i \quad \mbox{for} \quad i>0 \quad \mbox{and}\quad l_0-l_n=-k_0 \quad \mbox{in the circle case}.
 \end{equation}
 Moreover, if we apply Proposition \ref{P:restrictions} and Lemma \ref{L:unitaryequi} to
  projections $g_i(x)$ and $f_i(x)$, then we have the following restrictions on
  $l_i$.
  \begin{itemize}
  \item[(i)] If for some $x$ in $X_i$, $f_i(x)$ has finite rank, then
             \begin{equation}\label{E:eq2}
         l_i \geq - \rank(f_i(x)),
         \end{equation}
  \item[(ii)] If  for some $x$ in $X_i$, $1-f_i(x)$ has finite rank, then
  \begin{equation}\label{E:eq3}
  l_i \leq \rank(1-f_i(x)),
  \end{equation}
  \item[(iii)] If either end point of $X_i$ is infinite, then
  \begin{equation}\label{E:eq4}
  l_i=0.
  \end{equation}
  \end{itemize}
   We claim that these necessary conditions are also sufficient. To show this, we need a well--known identification  of a strongly continuous projection valued function on a topological space with a continuous field of Hilbert spaces \cite[252--253]{dixdou}: Given a separable Hilbert space $H$, there is a one to one
correspondence between complemented subfields $\mathcal{H}=((H_x)_{x \in X},\Gamma)$ of the
constant filed defined by $H$ over a paracompact space $T$ and strongly continuous
projection valued functions $p:X \mapsto B(H)$ since $H_x$ is given by $p(x)H$ and vice versa.  Thus two
continuous fields of Hilbert spaces defined by $p$ and $p'$ are
isomorphic if and only if there is a double strongly continuous
valued function $u$ on $T$ such that $uu^*=p'$ and $u^*u=p$ (for the notion of a continuous field of Hilbert spaces, see Section \ref{S:Appendix}). 

The following lemma plays a key role in the proof of Theorem \ref{T:liftingthm}. However, its proof is rather long, so we include the proof in Appendix (see Corollary \ref{C:subprojection}).
  \begin{lem}\label{L:subprojection}
If $X$ is a separable metric space such that $\dim(X) \leq 1$ and
$\mathcal{H}$ is a continuous field of Hilbert spaces over $X$ such
that $\dim (H_x) \geq n $ for every $x \in X$, then $\mathcal{H}$
has a trivial subfield of rank n. Equivalently, if $p$ is a strongly
continuous projection valued function on $X$ such that
$\rank(p(x))\geq n$ for every $x \in X$, then there is a norm
continuous projection valued function $q$ such that $q \leq p$ and
$\rank(q(x))=n$ for every $x \in X$.
\end{lem}
\begin{thm}\label{T:liftingthm}
  A projection $\mathbf{f}$ in $\mathcal{C}(I)$ represented by
  $(f_0,\cdots, f_n)$ is liftable to a projection in $M(I)$ if and only if there
  exist $l_0,\cdots,l_n$ satisfying above conditions
  (\ref{E:eq1}), (\ref{E:eq2}), (\ref{E:eq3}), (\ref{E:eq4}).
  \end{thm}
 \begin{proof}
  Given $l_i$'s satisfying (\ref{E:eq1}), (\ref{E:eq2}), (\ref{E:eq3}),
  (\ref{E:eq4}),
  we will show there exist $g_0, \cdots, g_n$
  such that $[g_i(x_i):g_{i-1}(x_i)]=0$ for $i>0$ and $[g_0(x_0):g_n(x_0)]=0$ in the circle case.

  First observe that if we have $g_{i}$'s such that $l_i=[g_i(x_i):f_i(x_i)]$,
  we have $[g_i(x_i):g_{i-1}(x_i)]=0 $ by (\ref{E:eq1}). Thus it is enough to
  show that there exist $g_0, \cdots, g_n$
  such that $[g_i(x_i):f_{i}(x_i)]=l_i$.
  \begin{itemize}
  \item[$l_i=0$]: Take $g_{i}=f_{i}$.
  \item[$l_i>0$]: By Lemma \ref{L:subprojection} the continuous field determined by $1-f_{i}$ has a
  rank $ l_i$ trivial subfield which is given by a projection valued function
  $q \leq 1-f_i$. So we take $g_i=f_i+q$.
  \item[$l_i<0$]: Similarly, the continuous field determined by $f_i$ has a rank
  $-l_i$ trivial subfield which is given by a projection valued function $q' \leq
  f_i$. So we take $g_i=f_i-q'$.
 \end{itemize}
  Then the conclusion follows from Proposition \ref{P:lifting}.
  \end{proof}

 \section{Homotopy classification of projections}\label{Homotopy}
 Applying the representation in section \ref{Lifting} to the $n\times n$ matrix algebra over $I$, we can
represent an
   element $\mathbf{f}$ in $M_{n}(\mathcal{C}(I))$ as $(f_1,f_2,\cdots, f_n)$
   where $f_{i}$'s are in $M_{n}(M(I))=M(M_{n}(I))$ and each of them is
   projection valued: $f_i(x)$ is in $M_n(B(H)) \simeq B(H^n)) \simeq B(H)$. In
    addition, since $f_i(x_i)-f_{i-1}(x_i) \in M_{n}(K) \simeq
    K$, we can also associate the integers $[f_i(x_i):f_{i-1}(x_i)]$
    for $i-1, \cdots, n$ by suppressing the identification.
   \begin{lem}\label{L:welldefine1}
     If $\mathbf{f} \in \mathcal{C}(I)$  where $I=C(-\infty,\infty)\otimes
   K$ has two  different local liftings $(f_1, \cdots f_n)$ and
    $(g_1,\cdots, g_n)$, then $\sum [f_i(x_i):f_{i-1}(x_i)]= \sum
    [g_i(x_i):g_{i-1}(x_i)]$
   \end{lem}
   \begin{proof}
    Note that
    \[
     [g_i(x_i):f_i(x_i)]+[f_i(x_i):f_{i-1}(x_i)]+[f_{i-1}(x_i):g_{i-1}(x_i)]
 =[g_i(x_i):g_{i-1}(x_i)].
  \]
  Equivalently,
  \[
       [f_i(x_i):f_{i-1}(x_i)]-[g_i(x_i):g_{i-1}(x_i)]=
       [f_i(x_i):g_i(x_i)]-[f_{i-1}(x_i):g_{i-1}(x_i)].
    \]
    Hence $$\sum [f_i(x_i):f_{i-1}(x_i)] - \sum
    [g_i(x_i):g_{i-1}(x_i)]=[f_n(x_n):g_n(x_n)] -[f_0(x_0):g_0(x_0)]. $$
    Since
    $f_n- g_n \in C_{0}(x_{n},\infty) \otimes K$ and
    $f_0- g_0 \in C_{0}(-\infty,x_{0}) \otimes K$,
    $[f_n(x_n):g_n(x_n)]=[f_0(x_0):g_0(x_0)]=0$ by Theorem \ref{T:homotopy}.
    Consequently, we have  $$\sum [f_i(x_i):f_{i-1}(x_i)]= \sum
    [g_i(x_i):g_{i-1}(x_i)]. $$
    \end{proof}

    Given two projections $\mathbf{p}, \mathbf{q}$ in $\mathcal{C}(I)$ let's assume there is a partial isometry $\pi(u)$ in $\mathcal{C}(I)$
    such that $\pi(u)^*\pi(u)=\mathbf{p}$ and
    $\pi(u)\pi(u)^*=\mathbf{q}$ for some $u$ in $M(I)$. If we take $(p_0, \cdots,
    p_n)$ and $(q_0, \cdots, q_n)$ as local liftings of $\mathbf{p}$
    and $\mathbf{q}$ respectively, using the same partition, we can
    have a representation of $u$ as $(u_0,\cdots,u_n)$. In fact,
    $u_i=q_iu|_{X_i}p_i$. Hence we
    have
    \[
     q_i - u_i{u_{i}}^{*},\quad p_i -{u_{i}}^{*}u_i  \in C(X_i)\otimes K.
    \]

    For any $x \in X_i$, we can view $u_i(x)$ as a Fredholm operator
    from $p_i(x)H$ to $q_i(x)H$, and thus we can define the
    Fredholm  index for $u_i(x)$ for each point $x$ in $X_i$.

   \begin{lem}\label{L:const}
    $\Ind(u_i(x))$ is constant on $X_i$.
   \end{lem}
    \begin{proof}
    Consider either end point of $X_i$ is not infinite. Fix a point $x_0$ and
   observe that $\ker(u_i(x_0))= E_{1}(p_i(x_0)-u_i(x_0)^*u_i(x_0))$. Similarly,
   $\ker(u_i(x_0)^*)=E_{1}(q_i(x_0)-u_i(x_0)u_i(x_0)^*)$.\\
    Note that $1$ is an isolated
   point in the spectrum of $p_i(x_0)-u_i(x_0)^*u_i(x_0)$, and therefore we can
   consider a neighborhood $\{1\}\cup [0, 1-\epsilon)$ for some $\epsilon >0$.
   Since $p_i-u_i^*u_i,
    q_i-u_iu_i^*$ is norm-continuous, there is a $\delta $ such that
    if $|x_0-x | < \delta$ then $\sigma(p_i(x)-u_i(x)^*u_i(x)) \subset
    O_{1}(\epsilon) \cup [0, 1-\epsilon) $ of
    $\sigma(p_i-u_i^*u_i)$ where $O_{1}(\epsilon)$ is the open ball
    centered at 1 with radius $\epsilon$. By the Riesz functional calculus (see the proof of Corollary
    \ref{C:homotopy}), $$\dim (E_{1}(p_i(x_0)-u_i(x_0)^*u_i(x_0)))=
    \sum_{0 \leq \lambda_j < \epsilon}
    \dim(E_{1-\lambda_j}(p_i(x)-u_i(x)^*u_i(x))).$$
    Similarly,$$\dim (E_{1}(q_i(x_0)-u_i(x_0)u_i(x_0)^*))=
    \sum_{0 \leq \nu_j < \epsilon}
    \dim(E_{1-\nu_j}(q_i(x)-u_i(x)u_i(x)^*)).$$
    Since
    $\dim E_{1-\lambda_i}(p_i(x)-u_i(x)^*u(x))=
    \dim E_{\lambda_i}(u_i(x)^*u_i(x))=
\dim E_{\lambda_i}(u_i(x)u_i(x)^*) $ $=
E_{1-\lambda_i}(q_i(x)-u_i(x)u(x)^*)$ for $\lambda_i>0$,
    $$\sum_{0 < \lambda_j < \epsilon}\dim(E_{1-\lambda_j}(p_i(x)-u_i(x)^*u_i(x)))=
    \sum_{0 < \nu_j <
    \epsilon}\dim(E_{1-\nu_j}(q_i(x)-u_i(x)u_i(x)^*)).$$
    Thus, we have
    \begin{align*}
    \Ind(u_i(x_0))&=\dim
    (E_{1}(p_i(x_0)-u_i(x_0)^*u_i(x_0)))-\dim
    (E_{1}(q_i(x_0)-u_i(x_0)u_i(x_0)^*)) \\
    &=\dim
    (E_{1}(p_i(x)-u_i(x)^*u_i(x)))-\dim
    (E_{1}(q_i(x)-u_i(x)u_i(x)^*)) \\
    &=\Ind(u_i(x)) \quad \mbox{for} \quad |x-x_0 | < \delta.
    \end{align*}
    Finally, the claim follows since each $X_i$ is connected.\\
    \end{proof}

    We will denote the index of $u_i$ on $X_i$ by $t_i$.
    If $u_i(x_i)=v_i|u_i(x_i)|$ is a polar decomposition of $u_i(x_i)$ in
    $B(H)$ and $u_{i-1}(x_i)=v_{i-1}|u_{i-1}(x_i)|$ is a polar
    decomposition of $u_{i-1}(x_i)$,
    then $$t_i=[p_i(x_i):v_i^*v_i]-[q_i(x_i):v_iv_i^*],$$
        $$t_{i-1}=[p_{i-1}(x_i):v_{i-1}^*v_{i-1}]-[q_{i-1}(x_i):v_{i-1}v_{i-1}^*].$$
    Since $p_{i}(x_i)-p_{i-1}(x_i), q_{i}(x_i)-q_{i-1}(x_i) \in K$, we
    have $u_i(x_i)-u_{i-1}(x_i) \in K$. Thus we can deduce that
      \[v_i -v_{i-1} \in K, \,\,
     v_i^*v_i-v_{i-1}^*v_{i-1},
     v_{i}v_{i}^*-v_{i-1}v_{i-1}^* \in K.\]

     From
    \begin{align*}
    [p_i(x_i):p_{i-1}(x_i)]&=
    [p_i(x_i):v_i^*v_i]+[v_i^*v_i:v_{i-1}^*v_{i-1}]+[v_{i-1}^*v_{i-1}:p_{i-1}(x_{i})],\\
    [q_i(x_i):q_{i-1}(x_i)]&=[q_i(x_i):v_iv_i^*]+[v_iv_i^*:v_{i-1}v_{i-1}^*]+[v_{i-1}v_{i-1}^*:q_{i-1}(x_{i})],
    \end{align*}
    we  have
    \[
       [p_i(x_i):p_{i-1}(x_i)]-[q_i(x_i):q_{i-1}(x_i)]=t_i-t_{i-1}+[v_i^*v_i:v_{i-1}^*v_{i-1}]-[v_iv_i^*:v_{i-1}v_{i-1}^*].
    \]
    Let $W, V$ be isometries such that $WW^*=v_i^*v_i,
     VV^*=v_{i-1}^*v_{i-1}$. Then $V'=v_{i-1}V, W'=v_iW$ are isometries such that
     $V'{V'}^{\ast}=v_{i-1}{v_{i-1}}^*,
     W'{W'}^{\ast}=v_iv_i^*$.
     \begin{align*}
     [v_iv_i^*:v_{i-1}v_{i-1}^*]&=\Ind({V'}^{\ast}W')\\
                                 &=\Ind(V^*v^*_{i-1}v_{i}W)\\
                                 &=\Ind(V^*{v^*_i}v_iW)  \quad  (v_{i-1} -v_i \in
                 K)\\
                                 &=\Ind(V^*WW^*W)\\
                                 &=\Ind(V^*W)\\
                                 &=[v^{\ast}_iv_i:{v^*_{i-1}}v_{i-1}].
     \end{align*}
     Thus, if we let $k_i=[p_i(x_i):p_{i-1}(x_i)]$,
     $l_i=[q_i(x_i):q_{i-1}(x_i)]$, then we have
     \begin{equation}\label{E:vonmurray}
      t_i -t_{i-1}= k_i-l_i.
     \end{equation}

     \begin{lem}\label{L:welldefine2}
    Suppose there is a partial isometry $\pi(u)$ in $\mathcal{C}(I)$ where $I=C(-\infty,\infty)\otimes
   K$ or $C(\mb{T})\otimes K$ such that $\pi(u)^*\pi(u)=\mathbf{p}$ and
    $\pi(u)\pi(u)^*=\mathbf{q}$ for some $u$ in $M(I)$. If $(p_0, \cdots,
    p_n)$ and $(q_0, \cdots, q_n)$ are local liftings of $\mathbf{p}$
    and $\mathbf{q}$ respectively, then $ \sum_{i=1}^n[p_i(x_i):p_{i-1}(x_i)]=\sum_{i=1}^n
    [q_i(x_i):q_{i-1}(x_i)]$, or $\sum_{i=1}^n[p_i(x_i):p_{i-1}(x_i)]+ [p_0(x_0):p_n(x_0)]=\sum_{i=1}^n
    [q_i(x_i):q_{i-1}(x_i)]+[q_0(x_0):q_n(x_0)]$ in the circle case.
     \end{lem}
     \begin{proof}
      By taking a sum in both sides of Eq. (\ref{E:vonmurray}), we have
      $\sum k_i -\sum l_i=t_n -t_0$. Since $q_0 - u_0{u_{0}}^{*},
     p_0 -{u_{0}}^{*}u_0
     \in C_{0}(-\infty,x_1]\otimes K$, and $q_n - u_n{u_{n}}^{*}, p_n -
     {u_{n}}^{*}u_n \in C_{0}[x_n,\infty)\otimes K$ we have
     $t_n=t_0=0$. 
     
     In the circle case, it follows by adding Eq. (\ref{E:vonmurray}) for $i=1,\dots, n+1$ modulo $n+1$.
      \end{proof}
    \begin{lem}\label{L:onetoone}
     Let $(p_0, \cdots,
    p_n)$ and $(q_0, \cdots, q_n)$ be local liftings of $\mathbf{p}$
    and $\mathbf{q}$ such that $\rank (1-q_{i}(x))=\rank (q_i(x))=\infty$
    for each $x$ in $X_i$.
    
    If $ \sum_{i=1}^n[p_i(x_i):p_{i-1}(x_i)]=\sum_{i=1}^n
    [q_i(x_i):q_{i-1}(x_i)]$, or  $\sum_{i=1}^n[p_i(x_i):p_{i-1}(x_i)]+ [p_0(x_0):p_n(x_0)]=\sum_{i=1}^n
    [q_i(x_i):q_{i-1}(x_i)]+[q_0(x_0):q_n(x_0)]$ in the circle case, then we can find a perturbation  $(q_0', \cdots,
    q_n')$ of $\mathbf{q}$ such that $[p_i(x_i):p_{i-1}(x_i)]=
    [q_i'(x_i):q_{i-1}'(x_i)]$ for $i=1,\dots, n $ or $[p_i(x_i):p_{i-1}(x_i)]=
    [q_i'(x_i):q_{i-1}'(x_i)]$for $i=1,\dots, n+1$ modulo $n+1$.
     \end{lem}
     \begin{proof}
      Let $[p_i(x_i):p_{i-1}(x_i)]=k_i, [q_i(x_i):q_{i-1}(x_i)]=l_i$.
      If $d_i=k_i-l_i$, note that
      $$ \sum[p_i(x_i):p_{i-1}(x_i)]=\sum
    [q_i(x_i):q_{i-1}(x_i)] \quad \text{if and only if} \quad \sum d_i =0. $$
      Let $q_0'=q_0$. Suppose that we have constructed $q_0',
\cdots, q_{i}'$ such that $[p_j(x_j):p_{j-1}(x_j)]=
    [q_j'(x_j):q_{j-1}'(x_j)]$ for $j=1,\cdots,i$ and
    $[q_{i+1}(x_{i+1}):q_{i}'(x_{i+1})]=l_{i+1}-\sum_{k=1}^{i}
    d_k$. \\
    If $d_{i+1}+\sum_{k=1}^{i} d_k >0$, let $q$ be a projection valued (norm continuous) function such that
    $q \leq 1-q_{i+1}$ and rank$(q(x))=d_{i+1}+\sum_{k=1}^{i}
    d_k$ which is possible since rank $(1-q_{i+1})(x)$ $\geq d_{i+1}+\sum_{k=1}^{i} d_k >0 $.\\
    Take $q_{i+1}'=q+q_{i+1}$. Then
    \begin{align*}
    [q_{i+1}'(x_{i+1}):q_{i}'(x_{i+1})]&=[q_{i+1}(x_{i+1}):q_{i}'(x_{i+1})]+[q(x_{i+1}):0] \\
                                              &=l_{i+1}-\sum_{k=1}^{i} d_k + d_{i+1}+\sum_{k=1}^{i} d_k\\
                                              &=l_{i+1}+k_{i+1}-l_{i+1}\\
                                              &=k_{i+1}
    \end{align*}
    \begin{align*}
    [q_{i+2}(x_{i+2}):q_{i+1}'(x_{i+2})]&=[q_{i+2}(x_{i+2}):q_{i+1}(x_{i+2})]+[0:q(x_{i+2})] \\
                                              &=l_{i+2} -( d_{i+1}+\sum_{k=1}^{i} d_k)\\
                                              &=l_{i+2}-\sum_{k=1}^{i+1} d_k
   \end{align*}

    If $d_{j+1}+\sum_{k=1}^{j} d_k < 0$, let $q$ be a projection valued (norm continuous) function such that
    $q \leq q_{i+1}$ and rank$(q(x))=-(d_{i+1}+\sum_{k=1}^{i}
    d_k)$ which is possible since rank $(q_{i+1})(x)$ $\geq -(d_{i+1}+\sum_{k=1}^{i} d_k) >0 $.\\
    Take $q_{i+1}'=q_{i+1}-q$. \\
     Note that
    $$[q_{i+1}'(x_{i+1}):q_{i}'(x_{i+1})]+[q(x_{i+1}):0]=
    [q_{i+1}'(x_{i+1}):q_{i}(x_{i+1})]$$
    Thus
    \begin{align*}
[q_{i+1}'(x_{i+1}):q_{i}'(x_{i+1})]&=l_{i+1}-\sum_{k=1}^{i}d_k + (d_{i+1}+\sum_{k=1}^{i} d_k)\\
                                         &=l_{i+1}+d_{i+1}\\
                                         &=k_{i+1}
    \end{align*}
    Also,
    \begin{align*}
    [q_{i+2}(x_{i+2}):q_{i+1}'(x_{i+2})]&=[q_{i+2}(x_{i+2}):q_{i+1}(x_{i+2})]+[q(x_{i+2}):0] \\
                                              &=l_{i+2} -( d_{i+1}+\sum_{k=1}^{i} d_k)\\
                                              &=l_{i+2}-\sum_{k=1}^{i+1} d_k
   \end{align*}
    By induction, we can get $q_{0}', \cdots, q_{n-1}'$ such
    that $[p_j(x_j):p_{j-1}(x_j)]=
    [q_j'(x_j):q_{j-1}'(x_j)]$ for $j=1,\cdots,n-1$
   as we want. \\
    Finally, since we also have $[q_{n}(x_{n}):q_{n-1}'(x_{n})]=l_{n}-\sum_{k=1}^{n-1}
    d_k=l_{n}+d_{n}=k_{n}$ from $\sum_{k=1}^{n-1}
    d_k+d_{n}=0$, we take $q_{n}'=q_{n}$.
    
    In the circle case, we perturb $q_n$ to $q'_n$ such that $[q_n'(x_n):q'_{n-1}(x_n)]=k_n$ and $[q_0(x_0):q'(x_0)]=l_0- \sum_{k=1}^n d^k=l_0+d_0=k_0$.
   \end{proof}
 Next is an analogous result that is more symmetrical.
   \begin{lem}\label{L:samerank}
    Let $(p_0, \cdots,
    p_n)$ and $(q_0, \cdots, q_n)$ be local liftings of $\mathbf{p}$
    and $\mathbf{q}$ such that $\rank(p_i(x))=\rank (q_i(x))=\infty$
    for each $x$ in $X_i$.\\
    If $ \sum[p_i(x_i):p_{i-1}(x_i)]=\sum
    [q_i(x_i):q_{i-1}(x_i)]$, or  $[p_i(x_i):p_{i-1}(x_i)]=
    [q_i'(x_i):q_{i-1}'(x_i)]$for $i=1,\dots, n+1$ modulo $n+1$, then we can find  perturbations  $(q_0', \cdots,
    q_n')$ of $\mathbf{q}$ and $(p_0', \cdots,
    p_n')$ of $\mathbf{p}$ such that $[p_i'(x_i):p_{i-1}'(x_i)]=
    [q_i'(x_i):q_{i-1}'(x_i)]$ for $i=1,\dots,n$, or $[p_i'(x_i):p_{i-1}'(x_i)]=
    [q_i'(x_i):q_{i-1}'(x_i)]$for $i=1,\dots, n+1$ modulo $n+1$ in the circle case. 
   \end{lem}
   \begin{proof}
    The proof proceeds as above with one exception: If $d_{i+1}+\sum_{k=1}^{i} d_k \geq
    0$, we make $p_{i+1}' \leq p_i$ rather than making $q_{i+1}'\geq
    q_{i}$.
  \end{proof}

 Recall that $K_i(\mc{C}(I))=K_{i+1}(I) \,(\mbox{mod}\,2)$. Thus
\begin{equation*}
    K_0(\mc{C}(I))=\begin{cases}
                  \mathbb{Z} &, \text{if $X=(-\infty,\infty)$}, \mb{T}=[0,1]/\{0,1\}\\
                  0          &, \text{otherwise}.
            \end{cases}
    \end{equation*}
    We want to analyze  these isomorphisms in a concrete way to compare $K_0$--classes of  projections in $\mc{C}(I)$
    where  $I=C(X)\otimes K$, and $X=(-\infty,\infty)$, or $X=\mb{T}=[0,1]/\{0,1\}$. For this
   we define a map $\chi:K_{0}(\mathcal{C}(I)) \to  \mathbb{Z}$
   as follows: Let $\alpha=[\mathbf{p}]-[\mathbf{q}]$ be an element of  $
     K_{0}(\mathcal{C}(I))$ and $(p_0, \cdots,
    p_n)$ and $(q_0, \cdots, q_n)$ be local liftings of $\mathbf{p}$
    and $\mathbf{q}$ respectively. Then $\chi(\alpha)=  \sum_{i=1}^{n}[p_i(x_i):p_{i-1}(x_i)]-\sum_{i=1}^{n}
    [q_i(x_i):q_{i-1}(x_i)]$ or $\sum_{i=1}^{n+1}[p_i(x_i):p_{i-1}(x_i)]-\sum_{i=1}^{n+1}
    [q_i(x_i):q_{i-1}(x_i)]$ with $x_{n+1}=x_0$ in the case $X=[0,1]/\{0,1\}$.

    \begin{thm}\label{T:main}
    The map $$ \chi:K_{0}(\mathcal{C}(I)) \to  \mathbb{Z} $$
     is an isomorphism.
    \end{thm}
    \begin{proof}
     $\chi$ is well-defined by Lemma \ref{L:welldefine1},
     \ref{L:welldefine2}.\\
      For injectivity of $\chi$,
     let $\alpha=[\mathbf{p}]-[\mathbf{q}]$ be an element of  $
     K_{0}(\mathcal{C}(I))$ such that $\chi(\alpha)=0$. If  $(p_0, \cdots,
    p_n)$ and $(q_0, \cdots, q_n)$ are local liftings of $\mathbf{p}$
    and $\mathbf{q}$ respectively, then $\mathbf{p}$ and $\mathbf{q}$ are Murray- von Neumann equivalent if and
    only if there is $\mathbf{u}$ in $\mathcal{C}(I)$ such that $\mathbf{u}^*\mathbf{u}=\mathbf{p}$ and
    $\mathbf{u}\mathbf{u}^*=\mathbf{q}$. To construct $\mathbf{u}$, it suffices to find a
    representation $(u_0,\cdots,u_n)$ such that
     $${u_i}^{\ast}u_i=p_i,\quad
    u_i{u_i}^{\ast}=q_i, \quad
    u_i(x_i)-u_{i-1}(x_i) \in K
     $$

    By replacing $\mathbf{p}$ with $\mathbf{p}\oplus 1_N \oplus 0_N$ and $\mathbf{q}$ with $\mathbf{p}\oplus 1_N \oplus 0_N$ for some $N$,
    we may assume that
    $\rank p_i(x)=\rank (1-p_i(x))=\rank q_i(x)= \rank (1-q_i(x))=\infty$.
    Moreover, if we let
    $k_i=[p_i(x_i):p_{i-1}(x_i)]$ and $l_i=[q_i(x_i):q_{i-1}(x_i)]$,
    by Lemma \ref{L:onetoone}, we may also assume that $k_i=l_i$ for each $i$.
    Since $\rank(p_i(x))=\rank(q_i(x))=\infty$, there is a (double) strongly
    continuous function $u_i$ on each $X_i$ such that ${u_i}^{\ast}u_i=p_i,
    u_i{u_i}^{\ast}=q_i$ because $\dim(X_i) < \infty$ (see Theorem \ref{T:triviality}). But,
     we do not claim $(u_0,\cdots,u_n) \in \mathcal{C}(I)$ because we have not yet achieved
    $u_i(x_i)-u_{i-1}(x_i) \in K$. \\
    On each $X_i$, we are going to change $u_i$ on a half-open interval including
    $x_i$ such that $u_i(x_i)-u_{i-1}(x_i) \in K$ and
    ${u_i}^*u_i=p_i, u_i{u_i}^{*}=q_i$. \\
    Since $\Ind(u_{i-1}(x_i))=0$, $$\Ind(q_i(x_i)u_{i-1}(x_i)p_i(x_i))=
   -l_i+\Ind(u_{i-1}(x_i))+k_i=0$$ and
    $q_i(x_i)u_{i-1}(x_i)p_i(x_i)-u_{i-1}(x_i) \in K$. Thus there is a
    $v_i \in B(H)$ which is a unitary from $p_i(x_i)H$ to $q_i(x_i)H$ such that
    $v_i-q_i(x_i)u_{i-1}(x_i)p_i(x_i) \in K$. Now ${u_i(x_i)}^*v_i$ is
    a unitary from $p_i(x_i)H$ to $p_i(x_i)H$. Combining  triviality of the continuous
    field of Hilbert spaces determined by $p_{i}$ and path connectedness of
    unitary group of B(H), we can form a path $\{v(t) : t\in [x_i,x]\}$ of unitaries
    such that $v(x_i)={u_i(x_i)}^{*}v_i$  and $v(x)=I_{p_{i}(x)H}$  for some $x \in
    X_i$. Let's take $w_i=u_iv$, then
    \begin{align*}
    w_i(x_i)-u_{i-1}(x_i) &= v_i
     - u_{i-1}(x_i) \in
    K \\
    {w_{i}}^*w_{i} &=v^*{u_i}^*u_iv=v^*p_iv=p_i \\
    w_{i}{w_{i}}^* &=u_ivv^*{u_i}^*=u_ip_i{u_i}^*=q_i
     \end{align*}
     Thus, we define
     \[
     u_i'=
             \begin{cases}
         w_i, \quad \text{on} \quad[x_i,x] \\
     u_i, \quad \text{on} \quad[x,x_{i+1}]
             \end{cases}
     \]
     Now $(u_0,u_1',\cdots,u_n')$ is what we want in the case $X=(-\infty, \infty)$, and by taking one more step $(u'_1,\dots,u'_n,u'_0)$ is what we want in the case $X=\mb{T}$.
    \end{proof}
    \begin{rem}
    In the above, in fact, we proved the following characterizations.
    \begin{itemize}
    \item[(i)]
    $[\mathbf{p}]=[\mathbf{q}]$ \quad if and only if $\quad \sum
     k_i=\sum l_i$.
    \item[(ii)]If $\rank(p_i(x))=\rank(q_i(x))=\infty $, then there exist $u_i'$s such that ${u_i}^*u_i=p_{i}$,
    $u_i{u_i}^*=q_i$ and $u_{i}(x_i)-u_{i-1}(x_i) \in K$
    if and only if $k_i=l_i$.
    \end{itemize}
    \end{rem}
    \begin{cor}
    $[\mathbf{p}]_0$ is liftable if and only if $p\oplus 1_n \oplus 0_n$
    is liftable.
    \end{cor}
    \begin{proof}
    Since $K_{0}(M(I))=0$, we know $[\mathbf{p}]_0$ is  liftable if and only
    if $[\mathbf{p}]_{0}=0$. The latter is equivalent to $\sum k_i
    =0$. Thus, if $[\mathbf{p}]_0$ is liftable, or $[\mathbf{p}\oplus 1_n \oplus
    0_n]$ is liftable, $\mathbf{p} \oplus 1_n \oplus
    0_n$ is liftable to a projection by Lemma \ref{P:lifting}. \\
    Conversely, if $\mathbf{p}$ is liftable to a projection, using the
    exponential map $\delta_0: K_{0}(\mathcal{C}(I)) \to K_{1}(C_0(-\infty,\infty)\otimes
    K)\simeq K_{0}(K)$, we can see that
    $[\mathbf{p}]_{0}=0$.
    \end{proof}
    \begin{prop}\label{P:vonmurray}
    $\mathbf{p} \sim \mathbf{q}$ in $\mathcal{C}(I)$ if and only if there exist (finite) $t_i$'s, $s_i$'s, $k_i$'s, $l_i$'s such that 
    \begin{itemize}
    \item[(i)]$\rank(p_i(x))=t_i+\rank(q_i(x))$,
    \item[(ii)]$t_i -t_{i-1}=k_i-l_i$,
    \item[(iii)]$t_i=0$ if $X_i$ contains infinite end-point
    \end{itemize}
    for some local liftings $(p_0, \cdots, p_n)$, $(q_0,\cdots,q_n)$
    of $\mathbf{p}$ and $\mathbf{q}$ respectively.
    \end{prop}
    \begin{proof}
    ``only if'':
    For the first condition, observe that if  $u_i(x)$ is a Fredholm operator
    from $p_i(x)H$ to $q_i(x)H$ and $\rank(p_i(x))$ is
    finite, then $\rank(q_i(x))$ is also finite and
    $\Ind(u_i(x))=t_i=\rank(p_i(x))-\rank(q_i(x))$.
    Moreover, if
    $\rank(p_i(x))$ is infinite, by applying the same reasoning to $q_i(x)$,
    $\rank(q_i(x))$ is also infinite. Therefore,
    $\rank(p_i(x))=\rank(q_i(x))+ t_i$ holds.

    The second condition follows from Eq.
    (\ref{E:vonmurray}).

    ``if'': If $\rank(p_i(x))=\rank(q_i(x))+t_i$ and $t_i \geq 0$ for each $i$, then $\rank(p_i(x))\geq
    t_i$, and therefore there is norm continuous projection valued
    function $p \leq p_i$ such that $\rank(p(x))=t_i$. Consider
    $p_i'=p_i-p$. It is easily checked that
    $\rank(p_i'(x))=\rank(q_i(x))$. So there is again a (double) strongly
    continuous function $u_i$ on each $X_i$ such that ${u_i}^{\ast}u_i=p_i',
    u_i{u_i}^{\ast}=q_i$. Furthermore,
    \begin{align*}
    {u_i}^*u_i -p_i & \in C(X_i)\otimes K \\
    u_i{u_i}^* -q_i & \in C(X_i)\otimes K \\
    \end{align*}
    implies that $u_i(x)$ is a Fredholm operator  and
    $\Ind(u_i(x))=t_i$ for every $x \in X_i$. \\
    Similarly, we can construct such a $u_i$ for the case $t_i < 0$
    since we can have a perturbation $q_i'$ of $q_i$ with
    $\rank(q_i'(x))=\rank(p_i(x))$.\\
     Note that \begin{align*}
     \Ind(q_{i}(x_i)u_{i-1}(x_i)p_i'(x_i) &=
     -[q_i(x_i):q_{i-1}(x_i)]+ \Ind(u_{i-1}(x_i))+ [p_i'(x_i):p_{i-1}(x_i)]\\
                                             &= -l_i + t_{i-1} + k_i-t_i\\
                                             &=0
     \end{align*}
     and $q_{i}(x_i)u_{i-1}(x_i)p_i'(x_i)- u_{i-1}(x_i) \in
    K$.\\
    As we have seen before in Theorem \ref{T:main}, this implies
    that we can have a perturbation $(u_0',\cdots u_{n}')$ of
    $(u_0,\cdots,u_n)$ such that $u_i'(x_i) -u_{i-1}'(x_i) \in
    K$.
    \end{proof}

    Since $\mathbf{p} \sim_u \mathbf{q}$ is equivalent to $\mathbf{p} \sim
    \mathbf{q}$ and $\mathbf{1-p} \sim \mathbf{1-q}$, we get the
    following statement immediately from Proposition
    \ref{P:vonmurray}.
    \begin{cor}\label{C:unitaryequi}
    $\mathbf{p} \sim_u \mathbf{q}$ in $\mc{C}(I)$ if and only if there exist (finite) $t_i$'s, $s_i$'s, $k_i$'s, $l_i$'s such that 
    \begin{itemize} 
    \item[(i)]$\rank(1-p_i(x))=s_i+\rank(1-q_i(x))$,
    \item[(ii)]$\rank(p_i(x))=t_i+\rank(q_i(x))$,
    \item[(iii)]$t_i-t_{i-1}=k_i-l_i$,
    \item[(iv)] $s_i+t_i=s_{i-1}+t_{i-1}$,
    \item[(v)] $t_i$ and $s_i$ are zeros if $X_i$ contains
    infinite end-point
    \end{itemize}
    for some local liftings $(p_0, \cdots, p_n)$, $(q_o,\cdots,q_n)$
    of $\mathbf{p}$ and $\mathbf{q}$ respectively.
    \end{cor}

    In general, $\mathbf{p} \sim_h \mathbf{q}$ in $\mc{C}(I)$ if and only
    if
    $\mathbf{u}\mathbf{p}\mathbf{u}^*=\mathbf{q}$ where $\mathbf{u}$
    is connected to 1 in the unitary group of $\mathcal{C}(I)$. It is said that a (non-unital) $C\sp{*}$-algebra $I$ has \emph{good index theory} if whenever
    $I$ is embedded as an ideal of in a unital $C\sp{*}$-algebra $A$
    and $u$ is a unitary in $A/I$ such that $\partial_1([u])=0$ in
    $K_0(I)$, there is a unitary in $A$ which lifts $u$.
    Equivalently, if $u\in \mc{U}(A/I)$, $\alpha \in K_1(A)$, and
    $\alpha$ lifts $[u]$, then $u$ lifts to
    $\mc{U}(A)$ (see \cite[2--3]{brped}). It was proved that any stable rank one $C\sp{*}$-algebra has good
    index theory by G.Nagy \cite{na}. Since $I$ has stable rank one,
    $I$ has good index theory. Also, recall that the unitary group of the multiplier algebra of
    a stable $C\sp{*}$-algebra is path connected (even contractible). Thus $\mathbf{p} \sim_h \mathbf{q}$ if and only
    if
    $\mathbf{u}\mathbf{p}\mathbf{u}^*=\mathbf{q}$ where $\mathbf{u}$
    has trivial $K_1$-class.

    \begin{cor}
$\mathbf{p} \sim_h \mathbf{q}$ in $\mc{C}(I)$ if and only if there exist (finite) $t_i$'s, $s_i$'s, $k_i$'s, $l_i$'s such that 
    \begin{itemize}
    \item[(i)] $\rank(1-p_i(x))=s_i+\rank(1-q_i(x))$,
    \item[(ii)]  $\rank(p_i(x))=t_i+\rank(q_i(x))$,
    \item[(iii)]$t_i-t_{i-1}=k_i-l_i$,
    \item[(iv)] $s_i+t_i=s_{i-1}+t_{i-1}=0$,
    \item[(v)] $t_i$ and $s_i$ are zeros if $X_i$ contains
    infinite point
    \end{itemize}
    for some local liftings $(p_0, \cdots, p_n)$, $(q_0,\cdots,q_n)$
    of $\mathbf{p}$ and $\mathbf{q}$ respectively.
    \end{cor}
    \begin{proof}
    ``Only if'': Since a unitary $\mathbf{u}$ which has trivial class in
    $K_1(\mathcal{C}(I))$ can be lifted to a unitary $u$ in $M(I)$,
    there is a unitary-valued function $u$ in $I$ such that
    $up_iu^*-q_i \in C(X_i)\otimes K$ for each i. Hence $q_i(x)u(x)p_i(x)$ and $(1-q_i(x))u(x)(1-p_i(x))$
    are Fredholm operators
    from $p_i(x)H$  to  $q_i(x)H $ and from $(1-p_i(x))H$
    to $(1-q_i(x))H$ respectively. Using matrix decomposition of $u_i$ and the fact
    that $(1-q_i)up_i, q_iu_i(1-p_i) \in
    C(X_i)\otimes K$(vanishing at infinite end-points), if we let
    $t_i=\Ind(q_i(x)u(x)p_i(x))$ and $s_i=\Ind((1-q_i(x))u(x)(1-p_i(x))$
    we can deduce that
   \begin{align*}
    t_i+s_i &=\Ind(q_i(x)u_i(x)p_i(x))+\Ind((1-q_i(x))u_i(x)(1-p_i(x))\\
            &=0
    \end{align*}
    ``If'': As we have shown in the proof of
    Proposition \ref{P:vonmurray}, from conditions (i),(ii),(iii),(iv) we can construct a (double) strongly
    continuous function $v_i$ on $X_i$ such that ${v_i}^{\ast}v_i=p_i' \leq p_i,
    v_i{v_i}^{\ast}=q_i$ with $v_i(x_i)- v_{i-1}(x_i) \in K$ and  a (double) strongly
    continuous function $w_i$ on $X_i$ such that $w_i{w_i}^{\ast}=q_i' \leq 1-q_i,
    {w_i}^{\ast}w_i=1-p_i$ with $w_i(x_i)-w_{i-1}(x_i) \in K$.  Then
    $u_i=v_i+w_i$
    is a function such that $u_ip_iu_i^*-q_i \in
    C(X_i)\otimes K$. Then let $\mathbf{u}$ be the element represented by the local representation
    $(u_0,\cdots,u_n)$.   Since $[\mathbf{u}]_1=\Ind(v_i+w_i)=0$
    via the map $K_1(\mc{C}(I)) \to K_0(I) \to K_0(K)=\mb{Z}$ which is induced from the evaluation at a point,  
     $\mathbf{u}\mathbf{p}\mathbf{u}^*=\mathbf{q}$  where $\mathbf{u}$ has trivial $K_1$--class, i.e., $\mathbf{p} \sim_h \mathbf{q}$.
   \end{proof}
   We conclude this section by constructing a projection in $\mc{C}(I)$ which does not lift stably but its $K_0$--class lifts where $I$ is the suspension and showing some examples which illustrate that any two equivalence notions do not coincide.
   For the following examples we frequently use Proposition \ref{P:constructproj} and Corollary \ref{C:constructproj} to construct projection valued function with desired rank properties.
   \begin{ex}
   Consider a partition $\{x_1,x_2 \}$ in $(-\infty,\infty)$. In
   other words, we divide $ (-\infty,\infty)$ into three intervals $X_0,X_1,X_2$. We
   can construct projection valued functions $p_0$, $p_1$, $p_2$ corresponding to $X_0,X_1,X_2$ respectively such that
 \begin{align*}
 &\rank(p_0(x_1))=8, \\
 &\rank(p_1(x_1))=\rank(1-p_1(x_2))=5,\\
 &\rank(1-p_2(x_2))=2 \quad \text{where $p_1(x_2)\leq p_2(x_2)$}
 \end{align*}
   Then we can easily check that $k_1=-3$, $k_2=3$. Suppose
   $\mathbf{p}$ which is represented by above $(p_0,p_1,p_2)$ is
   liftable to a projection in $M(I)$. By Theorem \ref{T:main}, we
   must have $l_0,l_1,l_2$ such that $l_0=0=l_2$, $l_1-l_0=-k_1=3$.
   However, since $\rank(1-p_1(x_2))=2$, we have a restriction $l_1\leq
   2$. This is a contradiction to $l_1-l_0=l_1=3$. However, since $\sum_i k_i=0$, its $K_0$--class does lift. If we also consider $q=\overbrace{p\oplus\cdots\oplus p}^n$ for any $n$, $k_i$'s for $q$ are just $n$--times of the above $k_i$'s. Then the same reasoning shows that $q$ is not liftable but its $K_0$-class is liftable. Thus $p$ is a projection in $\mc{C}(I)$ which does not lift stably. Note that if we change $\mathbf{p}$ into $\mathbf{p}\oplus 0$, $k_i$'s are preserved. Since $\rank(1-p_1\oplus 1)=\rank(1-(p_1\oplus
   0))=\infty$, we have no restriction on choosing $l_1$ as 3. Thus
   $\mathbf{p}\oplus0$ is liftable. 
   \end{ex}
   \begin{ex}
   Let $x_0,x_1,x_2$ be points on the interior of the circle $\mb{T}$. If we
   construct projection valued functions $p_0$, $p_1$, $p_2$ such that
 \begin{align*}
 &\rank(1-p_0(x_1))=3, \\
 &\rank(1-p_1(x_1))=1 \quad \text{where $p_0(x_1)\leq p_1(x_1)$}\\
 &\rank(p_1(x_2))=2,\\
 &\rank(p_2(x_2))=6, \\
 &\rank(p_2(x)) \quad \text{for some $x \in (x_2,x_0)$},\\
 &\rank(1-p_2(x_0))=1 \quad \rank(1-p_0(x_0))=7 \quad \text{where $p_0(x_o)\leq p_2(x_0)$}.
 \end{align*}
 then we can see that $k_1=2$, $k_2=4$, $k_0=-6$. If $\mathbf{p}$ which is represented by above $(p_0,p_1,p_2)$ is
   liftable to a projection in $M(I)$, by Theorem \ref{T:main}, we
   must have $l_0,l_1,l_2$ such that 
   $l_1-l_0=-k_1=-2$, $l_2-l_1=-k_2=-4$, $l_0-l_2=-k_0=6$. Note that $l_0 \leq \rank(1-p_0(x_1)=3$, $-2 \leq \rank(p_2(x))\leq l_2$. 
   Thus $2 \leq l_1=4+l_2 = l_1=l_0-k_1\leq 1$ which is a contradiction. Thus, $\mathbf{p}$
   cannot be liftable. If we change $mathbf{p}$ into $\mathbf{p}\oplus 1$, then $k_i$'s are preserved. Since $\mathbf{p} \oplus 1$ has infinite rank, there are no restrictions on lower bounds of $l_i$'s. So if we choose, for example, $l_0=0$, $l_1=-2$, $l_2=-6$, then $\mathbf{p}\oplus 1$ is   
   liftable.
  \end{ex}
  
   \begin{ex}
   Consider projection valued functions $p=1$ and $q$ on $[0,1]$
   which satisfies
   $\rank(q(x))=1$. Then $\pi(p)=\mathbf{p}$ and
   $\pi(q)=\mathbf{q}$ has trivial $K_0$-class.
   But we cannot find $t < \infty$ such that $\rank(p(x))= \rank(q(x))+t$. This shows that $[\mathbf{p}]=[\mathbf{q}]$ does not imply $\mathbf{p} \sim \mathbf{q}$.
   \end{ex}
   
   \begin{ex}
   Even if we are given non-trivial $K_0$-data for
   $\mathbf{p}$ and $\mathbf{q}$, we can find a example such that
   $\mathbf{p} \nsim \mathbf{q}$. Consider $(p_0,p_1)$ and
   $(q_0,q_1)$ on $(-\infty,\infty)$ such that
   $\rank(p_1(x_1))-\rank(p_0(x_1))=k_1=l_1=\rank(q_1(x_1))-\rank(q_0(x_1))$
   but $\rank(p_i(x)) \ne \rank(q_i(x))$. Moreover, this type of
   example give us projections $\mathbf{p}$ and $\mathbf{q}$ such
   that $\mathbf{p} \nsim \mathbf{q}$ but $\mathbf{1-p} \sim
   \mathbf{1-q}$. Define $p_0$ on $(-\infty,x_1]$ such that
   $\rank(p_0(x))=1$ and $p_1$ on $[x_1,\infty)$ such that
   $\rank(p_1(x))=2$. Similarly, we define $q_0$ and $q_1$ such that
   $\rank(q_0(x))=2, \rank(q_1(x))=3$. Note that
   $\rank(1-p_i(x))=\rank(1-q_i(x))=\infty$ but
   $\rank(p_i(x))\ne\rank(q_i(x))$. Since $t_0=t_1=0=s_0=s_1$ for this case,
   by Proposition \ref{P:vonmurray}, $\mathbf{1-p} \sim
   \mathbf{1-q}$ but $\mathbf{p} \nsim \mathbf{q}$. Thus we conclude $1-\mathbf{p} \nsim_u
   1-\mathbf{q}$ by Corollary \ref{C:unitaryequi}.
   \end{ex}
   
   \begin{ex}
   In $[0,1]$, if one division point is given, we can take
   projection valued function
   $p_0$, $q_0$ on $[0,x_1]$ such that
   \begin{equation*}
   \begin{cases}
   \rank(p_0(0))&=1, \\
   \rank(q_0(0))& =0, \\
   \rank(p_0(x))=\rank(q_0(x)) &= \infty \quad \text{if $x \ne 0$,}\\
   \rank(1-p_0(x_1))=\rank(1-q_0(x_1)) &= 2
\end{cases}
   \end{equation*}
Also, we can construct $p_1$, $q_1$ on $[x_1,1]$ such that
   \begin{equation*}
   \begin{cases}
   \rank(1-p_1(1))=\rank(1-q_1(1))=0&, \\
   \rank(1-p_1(x))=\rank(1-q_1(x))=2& \quad \text{if $x \ne 1$,}\\
   p_1(x_1)=p_0(x_1)&,\\
   q_1(x_1)=q_0(x_1)&
   \end{cases}
   \end{equation*}

  Then, we have $k_1=l_1=0$, $t_0=1$, $s_1=0$, thus if we take
  $t_1=1$, $s_0=0$ then $\mathbf{p} \sim_u \mathbf{q}$ but $\mathbf{p}$
  cannot be homotopic to $\mathbf{q}$ since $t_1+s_1 \ne 0$.
   \end{ex}

\section{Appendix}\label{S:Appendix}
In this section, we show some results about continuous fields of Hilbert spaces which were used implicitly or explicitly throughout this article. We refer the reader to \cite{dix} or \cite{dixdou} for a complete introduction of this notion. Recall that a continuous fields of Hilbert spaces over a topological space $X$ is a family of Hilbert spaces $(H_x)_{x \in X}$ together with a continuous structure $\Gamma$ consisting of vector sections $\xi$ in the product space $\prod_{x\in X}H_x$ satisfying the following two conditions:
\begin{enumerate}
\item The norm $x \to \|\xi(x)\|$ is continuous on $X$ for each $\xi \in \Gamma$.
\item The set $\{\xi(x) \mid \xi \in \Gamma \}$ is norm dense in $H_x$ for each $x \in X$. 
\end{enumerate}
 If $H_x$ is the same Hilbert space $H$ for every $x$, and $\Gamma$ consists of all continuous
mappings of $X$ into $H$, $\mathcal{H}$ is called a constant field.
A field (isometrically) isomorphic to a constant field is said to be
trivial. If $\mathcal{H}^{'}=((H^{'}_x)_{x\in X}, \Gamma^{'})$ is a
continuous filed of Hilbert spaces over $T$, $H^{'}_x$ is a closed
subspace of $H_x$ for each $t$, and $\Gamma^{'} \subset \Gamma$,
then $\mathcal{H}^{'}$ is called a subfield of $\mathcal{H}$.
Furthermore, we call $\mathcal{H}^{'}=((H^{'}_x)_{x\in X},
\Gamma^{'})$ a complemented subfield of $\mathcal{H}=((H_x)_{x\in
X},\Gamma)$ if there is a subfield
$\mathcal{H}^{''}=((H^{''}_x)_{x\in X}, \Gamma^{''})$ such that
$H^{'}_x \oplus H^{''}_x=H_x$ for every $x$. Also, we say that $\mathcal{H}$
is separable if $\Gamma$ has a countable subset $\Lambda$ such that
$\{\xi(x)\mid \xi \in \Lambda \}$ is dense in
$H_x$ for each $x$.
 As we already mentioned, complemented subfields over $X$ are in one--to--one correspondence with strongly continuous projection valued functions from $X$ to $B(H)$.  Related to this, we note that an absorbtion type theorem. 
 
 \begin{thm}\cite[Lemma 10.8.7]{dix}\label{T:triviality}
If $X$ is paracompact and of finite dimension, every separable
continuous field $\mathcal{H}=((H_x)_{ x \in X}, \Gamma)$ of Hilbert
spaces over $X$ such that $\dim(H_x)=\infty$ for every $x$ is
trivial. Thus two continuous fields $\mathcal{H}$, $\mathcal{H}'$
of Hilbert spaces over $X$ such that
$\dim{H_x}=\dim{H_{x}'}=\infty$ are isomorphic.
\end{thm}

\begin{thm}[\cite{dixdou}]\label{T:absorbtion}
Let $X$ be paracompact space and $\mc{H}$  be a separable
continuous field of Hilbert spaces over $X$. Then it is isomorphic
to a complemented subfield of a trivial field, and thus is
isomorphic to a continuous field defined by a strongly continuous
projection valued function $p:X \mapsto B(H)$.
\end{thm}
\begin{proof}
Since $\mathcal{H}^{(\infty)}$ is a trivial filed of infinite rank, then
$\mc{H}\oplus\mc{H}^{(\infty)}\cong \mc{H}^{(\infty)}$ by Theorem \ref{T:triviality}.
\end{proof}

Inspired by the above result, we want
to get a similar statement (Proposition \ref{P:isomorphic}) for two
continuous fields of Hilbert spaces whose fiber is possibly finite
dimensional.

 Given any continuous field $\mc{H}=((H_x)_{x\in X},\Gamma)$ and
any closed subset $A$ of $X$, we can get a continuous subfield
$\mc{H}^0$ such that
\begin{equation*}
H^0_x=\begin{cases}
 H_x, x \notin A\\
 0,  \, x \in A
 \end{cases}
 \end{equation*}

 Moreover, if $\mc{H}$ is separable and $X$ is separable metric space,
 then $\mc{H}^0$ is also separable. In fact, if $f$ is a section
 such that $f(x)=0$ for $x \in A$, for every $x \in X$, and
 $\epsilon>0$, there is a continuous section $g$ of $\mc{H}$ such that $\|f-g\|
 <\epsilon$ for some neighborhood $O$ of $x$. If we choose $\phi$ such that $\phi|_A=0$ and $\|\phi g-g \|<\epsilon$
 on $O$, it follows that $\|f-\phi g\| < 2\epsilon$. Hence $ \{\phi g \mid g \in \Gamma, \quad
 \phi:T\mapsto \mathbb{C} \,\,\text{such that}\,\, \phi|_A=0 \}$ define
 continuous vector fields of $\mc{H}^0$.

 The following Lemma \ref{L:sectionexist} and Proposition
 \ref{P:isomorphic} may be known by experts, but
 we shall give our proofs of these for the convenience of readers.
 Rather introducing a new notation, we also denote by $\dim$ the covering dimension of a topological
 space when it makes no confusion with the (linear) dimension of a
 Hilbert space.

\begin{lem}\label{L:sectionexist}
If $X$ is a separable metric space such that $\dim X \leq 1$, if
$\mathcal{H}$ is a continuous field of Hilbert spaces over $X$ such
that $H_x\ne0 $ for every $x\in X$, if $f$ is a continuous vector
field of $\mathcal{H}$, and if $\epsilon > 0$, then there is a
continuous vector field $g$ such that $\|g(x)-f(x)\|<\epsilon$ and
$g(x)\ne 0$ for every $x\in X$.
\end{lem}
\begin{proof}
 There is a countable open cover $\{U_n\}$ of $X$ such
that each $\mathcal{H}|_{U_n}$ has a non-vanishing section. By
paracompactness of $X$ there is a closed, locally finite refinement,
$\{F_n\}$, for $\{U_n\}$. By construction, for each $n$ we have
\begin{equation}\label{E:eqtn1}
\mathcal{H}|_{F_n}=\mathcal{L}_n \oplus K_n
\end{equation}
where $\mathcal{L}_n$ is a trivial subfield of rank one. Choose a
strictly increasing sequence $\{\epsilon_n\}$ of positive numbers
such that $\epsilon_n <\epsilon $. We will recursively construct
compatible sections $g_n$ over $F_n$ such that $g_n(x)\ne0$  and
$\|g_n(x)-f(x)\|\leq \epsilon_n$ for every $x \in F_n$. The local
finiteness ensures that the resulting global section $g$ is
continuous on $X$.\\
 To construct $g_n$, let
$A=F_n\cap(\bigcup_{k=1}^{n-1}F_k)$($A=\varnothing$
 if $n=1$). We have a non-vanishing section $g'$ on $A$ such that
 $\|g'(x)-f(x)\|<\epsilon_{n-1}$ for every $x\in A$, and we wish
 to extend $g'$ to $F_n$. Choose $\epsilon'$ and
 $\epsilon^{''}$ so that
 $\epsilon_{n-1}<\epsilon'<\epsilon^{''}<\epsilon_n$, and write
 $f=f^{1}\oplus f^{2}$, $g'=g^{1}\oplus g^{2}$ relative to the
 decomposition (\ref{E:eqtn1}).\\
 We first extend $g^2$ to all of $F_n$ so that $\|g_2(x)-f^2(x)\|\leq
 \epsilon'$ for every $x\in X$. To do this, let $h$ be an arbitrary
 extension of $g^2$ to $F_n$, which exists by Proposition 7 in
 \cite{dixdou}. Then let $B=\{x\in F_n\mid \|h(x)-f^2\|\geq
 \epsilon'\}$, and note that $A \cap B=\varnothing$. Let $\phi:F_n \mapsto
 [0,1]$ be a continuous function such that $\phi|_B = 1$ and
 $\phi|_A=0$, and take $g^2=\phi f^2+(1-\phi)h$.\\
 Next, we extend $g^1$ to a section $k$ on all of $F_n$ so that
 $\|k(x)-f^1(x)\|^2 + \|g^2(x)-f^2(x)\|^2 \leq (\epsilon^{''})^2$.
 It will be convenient to identify sections of $\mathcal{L}_n$ with
 complex valued functions and define
 \begin{equation*}
 \phi(s,z)=\begin{cases}
           z, \quad\text{if $|z|\leq s$} \\
           s\frac{z}{|z|}, \, \text{if $|z|>s$}
 \end{cases}
 \end{equation*}
 for $s>0$ and $z\in \mathbb{C}$. Thus $\phi$ is continuous on
 $(0,\infty)\times\mathbb{C}$. Now extend the function $g^1-f^1$ to
 $l$ on $F_n$ and let
 $k(x)=f^1(x)+\phi(\sigma(x),l(x))$, where $\sigma(x)=((\epsilon^{''})^2-\|g^2(x)-f^2(x)
 \|^2)^{1/2}$.\\
 Finally, we must modify $k$ to obtain the non-vanishing property
 without changing $k|_A$. Let $C={x\in F_n \mid g^2(x)=0}$, $D={x\in C \mid
 k(x)=0}$, and $\delta=\frac{\epsilon_n-\epsilon^{''}}{2}$. Since $D \cap
 A=\varnothing$, there is an open neighborhood $V$ of $D$ such that
 $\overline{V}\cap A=\varnothing $. Let $G=\overline{V}\cap C$ and
 $E=((\overline{V}-V)\cap C) \cup \{x\in G \mid \|k(x)\|=\delta \} $
By dimension theory, there is a non-vanishing continuous function
$r$ on $G$ such that $r|_E=k|_E$. Then  define $g^1|_C(x)$ by
\begin{equation*}
g^1(x)=\begin{cases}
    \phi(\delta,r(x)) \quad \text{if $x\in V\cap C$ and
    $\|k(x)\|<\delta$}\\
    k(x) \quad \text{otherwise on $C$}
\end{cases}
\end{equation*}
Note that if $x_n \in V \cap C$, $\|k(x_n)\| < \delta$ and if $x_n
\to x $ for some $x$ not satisfying these conditions, then
$r(x)=k(x)$ and $\|k(x)\|=\delta$ Thus $$g^1(x_n) \to
\phi(\delta,r(x))=k(x)$$ This implies $g^1$ is continuous on $C$ and
$\|g^1-k \| \leq 2\delta$.\\
Now $g^1$ is defined on $C \cup A$, and we extend $g^1$ to $F_n$ so
that $\|g^1(x)-k(x) \|< 2\delta$ for every $x \in F_n$. This can be
done as in the previous paragraph. Then $g_n=g^1 \oplus g^2$
satisfies all required properties.
\end{proof}

\begin{prop}\label{P:isomorphic}
If $X$ is a separable metric space such that $\dim X \leq 1$, and if
$\mathcal{H}$ and $K$ are separable continuous fields of
Hilbert spaces over $X$ such that $\dim H_x =\dim K_x$ for every $x
\in X$, then $\mathcal{H} \cong K$.
\end{prop}
\begin{proof} For $n=1,\cdots,$ let $U_n=\{x\in X \mid \dim H_x \geq
n \}$, an open set. Let $\mathcal{L}_n$ be a trivial line bundle
over $U_n$, extended by zero as as to be a continuous field over
$X$. Thus continuous sections of $\mathcal{L}_n$ can be identified
with continuous complex-valued functions on $X$ which vanished on $X
\smallsetminus U_n$. We will show that $\mathcal{H} \equiv
\bigoplus_1^{\infty} \mathcal{L}_n $. Since the same argument
applies to $K$, the result follows.

To do this we construct recursively a sequence $\{e_n\}$ such that
$e_n$ is a continuous section of $\mathcal{H}|_{U_n}$, such that
$$\|e_n(x)\|=1 \quad \text{for every $x \in U_n$}, <e_n(x),e_m(x)>=0 \quad \text{if
$n<m$ and $x \in U_m$}.$$ We will impose additional conditions on
the $e_n$'s, but first we point out that any such $e_n$'s give rise
to complemented subfield $\mathcal{M}_n$, where
$(\mathcal{M}_n)_x=\Span(e_n(x))$. Moreover if we write
$\mathcal{H}=\mathcal{H}_n'\oplus \mathcal{M}_1\oplus \cdots
\oplus \mathcal{M}_n$, then $\dim (H_n')_x= \maxi(\dim H_x-n,0)$.
It is enough to consider the case $n=1$. Note that for any
continuous section $f$ of $\mathcal{H}$, we have the function $c_1$
defined by \begin{equation*}
 c_1(x)=\begin{cases}
    <f(x),e_1(x)> &, \quad x \in U_1 \\
    0              &, \quad x \notin U_1
 \end{cases}
\end{equation*}
is continuous, since $| <f(x),e_1(x)> | \leq \| f(x)\|$ and $f$ is
vanishes on $X-U_1$. Thus $g = f- c_1 e_1$ may be regarded as a
continuous section of $\mathcal{H}$ such that
$\|g(x)\|^2=\|f(x)\|^2-\|c_1(x)\|^2$, this implies that
$\mathcal{H}_1'=\mathcal{M}_1^{\perp}$ is indeed a continuous
subfield of $\mathcal{H}$. The dimension formula given above is
obvious, and it follows that $\{x\mid \dim (\mathcal{H}_1')_x
\geq n \}=U_{n+1}$. Now we proceed by induction on $n$, working with
$\mathcal{H}_1'$ instead of $\mathcal{H}$.

Let $\{f_m \}$ be a sequence of continuous sections of $\mathcal{H}$
such that $\mathcal{H}_x= \overline{\Span (f_n(x))}$ for each $x$.
Let $\{ g_n \}$ be a sequence which includes each $f_m$ infinitely
many times. We choose the $e_n$'s so that for each $n$ and $x$, the
projection of $g_n(x)$ on $(\mathcal{H}_n')_x$ has norm at most
$1/n$. If this is so, then $\mathcal{H}\cong
\bigoplus\mathcal{M}_n$; and since $\mathcal{M}_n\cong
\mathcal{L}_n$, the result follows.

Assume $e_k$ has already been constructed for $k < n$. Let $h$ be
the $\mathcal{H}_{n-1}'$ component of $g_n$. Apply Lemma
\ref{L:sectionexist} to $\mathcal{H}_{n-1}'|_{U_n} $ and
$h|_{U_n}$ with $\epsilon=1/n$.(Recall that $U_n=\{ x \mid
(H_{n-1}')_x \ne 0\}$.) Thus we obtain a non-vanishing section
$l$ on $U_n$, such that $\|h(x)-l(x) \| \leq 1/n$ for every $x$. If
$e_n(x)=\dfrac{l(x)}{\|l(x)\|}$, then $e_n$ satisfies all our
requirements.
\end{proof}

\begin{cor}\label{C:subprojection}
If $X$ is a separable metric space such that $\dim(X) \leq 1$ and
$\mathcal{H}$ is a continuous field of Hilbert spaces over $X$ such
that $\dim (H_x) \geq n $ for every $x \in X$, then $\mathcal{H}$
has a trivial subfield of rank n. Equivalently, if $p$ is a strongly
continuous projection valued function on $X$ such that
$\rank(p(x))\geq n$ for every $x \in X$, then there is a norm
continuous projection valued function $q$ such that $q \leq p$ and
$\rank(q(x))=n$ for every $x \in X$.
\end{cor}
\begin{proof} Suppose $\dim H_x \geq n$. Let $\mathcal{H}_0$ be the
trivial field of rank $n$ and $K=((K_x)_{x\in X}, \Gamma)$
be a continuous field of Hilbert spaces such that $\dim K_x= \dim
H_x - n$: Let $m(x)=\dim H_x -n$. Since $m$ is lower
semi-continuous, for each k $O_k=\{x\mid m(x)\geq k \}$ is open so
that we can construct a continuous field $\mathcal{H}_k$ of Hilbert
spaces such that
\begin{equation*}
\dim H^{k}_x=
\begin{cases}
1 \quad\text{if $x\in O_k$}\\
0 \quad \text{otherwise}
\end{cases}
\end{equation*}
Now take $K$ as $\bigoplus_k \mathcal{H}_k$.\\  If we let
$\mathcal{H}'=\mathcal{H}_0\oplus K$, $\dim H'_x=\dim
H_x$ from the above. Thus the conclusion follows from Proposition
\ref{P:isomorphic}.
\end{proof}

\begin{cor}
If $X$ is a separable metric space such that $\dim(X) \leq 1$ and
$H_x \ne 0$ for every $x \in X$, then there is a vector field $v$
such that $v(x)\ne 0$ for every $x \in X$.
\end{cor}

The following facts are useful and well-known. But we include proofs
for the sake of completeness.
\begin{prop}\label{P:constructproj}
Given a closed interval $X$ and a point $a$ in $X$, there is a
strongly continuous projection valued function $p$ on $X$ such that
\begin{equation*}
\rank(p(x))= \begin{cases}
 1 \quad \text{if $x \ne a$,}\\
 0 \quad \text{if $x=a$}
 \end{cases}
\end{equation*}
\end{prop}
\begin{proof}
 Since there is a separable continuous field $\mc{H}$
such that
\begin{equation*} \dim(H_x)= \begin{cases}
 1 \quad \text{if $x \ne a$,}\\
 0 \quad \text{if $x=a$}
 \end{cases},
\end{equation*} it follows from Theorem \ref{T:absorbtion}.
\end{proof}

\begin{rem}
In fact, we can prove the above proposition using elementary
arguments. Without loss of generality, we can assume $X=[0,1]$,
$a=0$. Since any separable Hilbert space is isomorphic to
$L^{2}[0,1]$, it is enough to construct such a function in
$L^{2}[0,1]$. We view a characteristic function corresponding to
$[0,t]$ as a vector $v_t$. As $t$ goes to 0, the function defined by
$v_t$ goes to 0 weakly in $L^{2}[0,1]$ by the Lebesgue Dominated
Convergence Theorem. Now it is easy to check that the map $t
\mapsto$ rank one projection $\frac{1}{\|v_t\|^2}\theta_{v_t,\,v_t}$
is a strongly continuous projection valued function $p$ such that
$p(0)=0$.
\end{rem}
\begin{cor}\label{C:constructproj}
Given a closed interval $X$ and a point $a$ in $X$, there is a
strongly continuous projection valued function $p$ such that
\begin{equation*}
\rank(p(x))= \begin{cases}
 \infty \quad \text{if $x \ne a$,}\\
 0 \quad \text{if $x=a$}
 \end{cases}
\end{equation*}
\end{cor}
\begin{proof}
 Let $p_i$ be the function on $X$ such that
\begin{equation*} \rank(p_i(x))= \begin{cases}
 1 \quad \text{if $x \ne a$,}\\
 0 \quad \text{if $x=a$}
 \end{cases}
\end{equation*} for each $i$. Then the map $x \mapsto
 \bigoplus_i p_i(x) \in \bigoplus_i L^{2}[0,1]$ is the required map.
\end{proof}
   
   \end{document}